\documentclass[12pt]{elsarticle}
\usepackage[top=1.0cm,bottom=1.0cm,left=2.5cm,right=2.5cm,
includehead,includefoot]{geometry}

\usepackage[CJKbookmarks=true,
            bookmarksnumbered=true,
            bookmarksopen=true,
            colorlinks=true,
            citecolor=blue,
            linkcolor=blue,
            anchorcolor=green,
            urlcolor=blue
            ]{hyperref}
\usepackage{hyperref}
\usepackage{color}

\usepackage{amsmath}
\usepackage{amssymb}
\usepackage{amsthm}
\usepackage{mathbbold}
\usepackage{mathrsfs}
\usepackage{bm}
\usepackage{diagrams}
\usepackage{graphicx}
\usepackage{epstopdf}
\usepackage{subfigure}
\usepackage{caption2}
\usepackage{booktabs}
\usepackage{multirow}
\usepackage{dcolumn}
\usepackage{ifthen}
\ifpdf
  \DeclareGraphicsExtensions{.eps,.pdf,.png,.jpg}
\else
  \DeclareGraphicsExtensions{.eps}
\fi

\biboptions{square,numbers,sort&compress}

\journal{ }

\begin{document}

\renewcommand{\figurename}{Fig.}
\renewcommand{\captionlabeldelim}{.}
\newcommand{\ucite}[1]{\scalebox{1.3}[1.3]{\raisebox{-0.65ex}{\cite{#1}}}}

\newcommand{\ud}{\mathrm{d}}
\newtheorem{defn}{Definition~}[section]
\newtheorem{lem}{Lemma~}[section]
\newtheorem{prop}{Proposition~}[section]
\newtheorem{thm}{Theorem~}[section]
\newtheorem{cor}{Corollary~}[section]
\newtheorem{conj}{Conjecture~}[section]
\newtheorem{exmp}{Example~}[section]
\newtheorem{rem}{Remark~}[section]
\newtheorem{supp}{Suppose~}[section]

\newcommand{\fv}{true}

\begin{frontmatter}



\title{Sparse residual tree and forest}


\author{Xin Xu}
\ead{xuxin103@163.com,xx2@princeton.edu}

\author{Xiaopeng Luo}
\ead{luo\_works@163.com,xiaopeng@princeton.edu}

\address{Department of Chemistry, Princeton University, Princeton, NJ 08544, USA}

\begin{abstract}
Sparse residual tree (SRT) is an adaptive exploration method for multivariate scattered data approximation. It leads to sparse and stable approximations in areas where the data is sufficient or redundant, and points out the possible local regions where data refinement is needed. Sparse residual forest (SRF) is a combination of SRT predictors to further improve the approximation accuracy and stability according to the error characteristics of SRTs. The hierarchical parallel SRT algorithm is based on both tree decomposition and adaptive radial basis function (RBF) explorations, whereby for each child a sparse and proper RBF refinement is added to the approximation by minimizing the norm of the residual inherited from its parent. The convergence results are established for both SRTs and SRFs. The worst case time complexity of SRTs is $\mathcal{O}(N\log_2N)$ for the initial work and $\mathcal{O}(\log_2N)$ for each prediction, meanwhile, the worst case storage requirement is $\mathcal{O}(N\log_2N)$, where the $N$ data points can be arbitrary distributed. Numerical experiments are performed for several illustrative examples.
\end{abstract}

\end{frontmatter}


\section{Introduction}
\label{SRTF:s1}

Multivariate scattered data approximation problems arise in many areas of engineering and scientific computing. In the last five decades, radial basis function (RBF) methods have gradually become an extremely powerful tool for scattered data. This is not only because they possess the dimensional independence and remarkable convergence properties (see, e.g., \cite{RiegerC2010_RBFSamplingInequalities,WendlandH2005B_ScatteredData,
WendlandH2005_RBFSamplingInequalities,WuZ1993A_ErrorEstimatesRBF,LuoXP2014M_ReproducingKernelHDMR}), but also because a number of techniques, such as multipole (far-field) expansions \cite{BeatsonRK1999M_FarFieldExpansions,
WendlandH2005B_ScatteredData}, multilevel methods of compactly supported kernels \cite{FloaterM1996M_MultilevelRBF,GeorgoulisE2012M_MultilevelRBF,
WendlandH2005B_ScatteredData,XuX2015A_MeshlessPDE} and partition of unity methods \cite{BabuskaI1997M_PU,LarssonE2017M_RBF&PU,WendlandH2005B_ScatteredData}, have been proposed to reduce both the condition number of the resulting interpolation matrix and the complexity of calculating the interpolant. These techniques are, of course, very important in practice, however, in contrast to the stability and efficiency, maybe the later question is the most crucial one for a general representation of functions, that is, how to accurately capture and represent the intrinsic structures of a target function, especially in high dimensional space.

More specifically, when the data and the expected accuracy are given, we usually do not know at all whether the data is redundant or insufficient for the target function. So it is necessary to consider the following three questions:
\begin{itemize}
  \item Whether the current data is just right to reach the expected accuracy?
  \item How to establish a sparse approximation by ignoring the possible redundancy?
  \item How to update the approximation by replenishing the possible insufficiency?
\end{itemize}
It is often difficult to distinguish between data insufficiency and redundancy, and they could in fact exist simultaneously in different local regions.

Sparse residual tree (SRT) is developed for the purpose of representing the intrinsic structure of arbitrary dimensional scattered data. SRT is based on both tree decomposition and adaptive radial basis function (RBF) explorations. For each child a concise and proper RBF refinement, whose shape parameter is related to the current regional scale, is added to the approximation by minimizing the $2$-norm of the residual inherited from its parent; then the tree node will be further split into two according to the updated residual; and this process finally stops when the data is insufficient or the expected accuracy is reached.

The word ``sparse" here has two meanings: (i) the RBF exploration applies only to a sparse but sufficient subset of the current data, which is to ensure the efficiency of the training process; (ii) the centers of the RBF refinement are also sparse relative to the sparse subset, which is to ensure the efficiency of the prediction process. Thus, on the one hand, SRT provides sparse approximations in areas where the data is sufficient or redundant, and on the other hand, SRT points out the possible local regions where data refinement is needed. In order to ensure stability, the condition number is strictly controlled for every refinement. Furthermore, SRT also yields the excellent performance in terms of efficiency. Similar to most typical tree-based algorithms \cite{BentleyJ1975M_kdTree,FriedmanJ1977M_BallTree}, the worst case time complexity of SRTs is $\mathcal{O}(N\log_2N)$ for the initial training work and $\mathcal{O}(\log_2N)$ for each prediction; and the worst case storage requirement is $\mathcal{O}(N\log_2N)$, where the $N$ data points can be arbitrary distributed. The training process can be accelerated using multi-core architectures. This hierarchical parallel algorithm allows one to easily handle ten millions of data points on a personal computer, or much more on a computer cluster.

Although there are some different attempts to combine tree structures and RBF methods in the field of machine learning (see, e.g., \cite{AkbilgicO2014M_Tree&RBF,
FeiB2006M_Tree&SVM,HadyM2010M_Tree&SVM}), they have not paid any attention to their convergence. In fact, similar to multilevel methods \cite{WendlandH2005B_ScatteredData}, these combinations do not always guarantee convergence. Most of the previously used error estimates for RBF interpolation depend on the so-called power function \cite{NarcowichF2003A_RefinedErrorEstimatesRBF,
SchabackR1995A_ErrorEstimates&ConditionNumbersRBF,WendlandH2005B_ScatteredData,
WuZ1993A_ErrorEstimatesRBF}. But recently, sampling inequalities have become a more powerful tool in this respect, and not limited to the case of interpolation \cite{NarcowichF2005A_RBFSamplingInequalities,MadychW2006A_SamplingInequalities,
RiegerC2010_RBFSamplingInequalities,WendlandH2005_RBFSamplingInequalities}. Sampling inequalities describe the fact that a differentiable function whose derivatives are bounded cannot attain large values if it is small on a sufficiently dense discrete set. Together with the stability of the least squares framework for residual trees, we prove that a SRT based on arbitrary basis functions leads to algebraic convergence orders for finitely smooth functions. Further combining the appropriate embeddings of certain native spaces, we also prove that the Gaussian or inverse multiquadric based SRT leads to exponential convergence orders for infinitely smooth functions.

Since the SRT approximation is actually piecewise smooth, the error of each piece is significantly larger near the boundary. And the sparse residual forest (SRF), which is a combination of SRT predictors with different tree decompositions, is specifically designed to improve this situation. For all SRTs in the SRF, the splitting method of each SRT depends on the values of a random vector sampled independently and with the same distribution. This provides an opportunity to avoid those predictions with large squared deviations and to use the average value of the remaining predictions to enhance both stability and convergence. In practice, SRFs composed of a small number of SRTs perform quite well than individual SRTs; and in theory, similar to random forests \cite{BreimanL2001A_RandomForest}, the error for SRFs converges with probability $1$ to a limit as the number of SRTs in the SRF becomes large. It is more efficient and accurate than the traditional partition of unity method for overcoming the boundary effect of the error.

The remainder of the paper is organized as follows. After appropriate notation and
preliminaries are introduced in section \ref{SRTF:s2}, section \ref{SRTF:s3} and section \ref{SRTF:s4} give the frameworks of the SRTs and SRFs, respectively, and the stability, convergence and complexity of the SRT algorithm are discussed in section \ref{SRTF:s5}. A series of numerical experiments is given in section \ref{SRTF:s6}. In section \ref{SRTF:s7}, we draw some conclusions on the new method presented in this work and discuss possible extensions.

\section{Notation and Preliminaries}
\label{SRTF:s2}

Throughout the paper, $e$ denotes Euler's constant, the space dimension $d\in\mathbb{N}$, the domain $\Omega\subset\mathbb{R}^d$ is convex, $f:\Omega\to\mathbb{R}~\textrm{or} ~\mathbb{C}$ is a given target function, $X=\{x_i\}_{i=1}^N \in\Omega$ is a set of pairwise distinct interpolation points with the fill distance
\begin{equation}\label{SRTF:eq:h}
  h:=h_{X,\Omega}:=\sup_{x\in\Omega}\min_{x_i\in X}\|x-x_i\|_2,
\end{equation}
and $f_X=(f(x_1),\cdots,f(x_N))^\mathrm{T}$ are known function values.

\begin{rem}
It is worth noting that $\Omega$ can also be extended to a finite union of convex domains, thereby $\Omega$ is bounded with Lipschitz boundary and satisfies an interior cone condition. In this case, we can first deal with these convex domains separately and then combine them into a meaningful whole by a suitable partition of unity, see section \ref{SRTF:s6} for examples.
\end{rem}
We will focus mainly on the Gaussian kernel $G(x)=G_\delta(x):=e^{-\delta^2\|x\|_2^2}$, where $\delta>0$ is often called the \emph{shape parameter}. Suppose that $\Omega'\subseteq\Omega$ is also convex, then for the subset $X'=\{x'_i\}_{i=1}^{N'}=X\cap\Omega'$ and selected centers $X''=\{x''_i\}_{i=1}^{N''}\subseteq X'$, where $N''\leqslant N'\leqslant N$, an Gaussian RBF approximation $s$ is required to be of the form
\begin{equation*}
  s_f(x,\alpha)=\sum_{j=1}^{N''}\alpha_jG(x-x''_j),~~~x\in\Omega'
\end{equation*}
with unknown coefficients $\alpha=(\alpha_1,\cdots,\alpha_{N''})^\mathrm{T}$.
Consider the following least squares (LS) problem
\begin{equation}\label{SRTF:eq:LS}
  \min_{\alpha\in\mathbb{R}^{N''}}\sum_{i=1}^{N'}
  \Big(s_f(x'_i,\alpha)-f(x'_i)\Big)^2.
\end{equation}
It is worth noting here that we consider the case of $N''\ll N'$ as a sparse approximation, and \eqref{SRTF:eq:LS} can be rewritten in matrix form as
\begin{equation}\label{SRTF:eq:LSM}
  \min_{\alpha\in\mathbb{R}^{N''}}\|\Phi_{X',X''}\alpha-f_{X'}\|_2^2,
\end{equation}
where the matrix $\Phi_{X',X''}\in\mathbb{R}^{N'\times N''}$ is generated by the Gaussian kernel $G(x)$. Suppose $\Phi_{X',X''}$ have a $QR$ decomposition $\Phi_{X',X''}=QR$, where $Q\in\mathbb{R}^{N'\times N''}$ has orthonormal columns and $R\in\mathbb{R}^{N''\times N''}$ is upper triangular, then the problem \eqref{SRTF:eq:LS} has a unique solution $\alpha^*=\Phi_{X',X''}^{-1}f_{X'}=R^{-1}Q^\mathrm{T}f_{X'}$, where $R$ and $Q^\mathrm{T}f_{X'}\in\mathbb{R}^{N''}$ can be recursively obtained without computing $Q$ by Householder transformations \cite{GolubG2013B_MatrixComputations}.

We shall consider functions from certain Sobolev spaces $W_p^k(\Omega)$ with $1\leqslant p<\infty$ and native spaces of Gaussians, $\mathcal{N}_G(\Omega)$, respectively. The Sobolev space $W_p^k(\Omega)$ consists of all functions $f$ with distributional derivatives $D^\gamma f\in L_p(\Omega)$ for all $|\gamma|\leqslant k$, $\gamma\in\mathbb{N}_0^d$. Associated with these spaces are the (semi-)norms
\begin{equation*}
  |f|_{W_p^k(\Omega)}=\left(\sum_{|\gamma|=k}
  \|D^\gamma f\|_{L_p(\Omega)}^p\right)^{1/p}
  ~~\textrm{and}~~
  \|f\|_{W_p^k(\Omega)}=\left(\sum_{|\gamma|\leqslant k}
  \|D^\gamma f\|_{L_p(\Omega)}^p\right)^{1/p}.
\end{equation*}
For the Gaussian kernel $G(x)=e^{-\delta^2\|x\|_2^2}$ the native space on $\mathbb{R}^d$ is given by
\begin{equation*}
  \mathcal{N}_G(\mathbb{R}^d)=\left\{f\in C(\mathbb{R}^d)\cap
  L_2(\mathbb{R}^d):\|f\|_{\mathcal{N}_G}:=\left(\int_{\mathbb{R}^d}
  |\hat{f}(\omega)|^2e^{\frac{\|\omega\|_2^2}{4\delta^2}}
  \ud\omega\right)^\frac{1}{2}<\infty\right\},
\end{equation*}
further, the native space $\mathcal{N}_G(\Omega)$ on a bounded domain $\Omega$ is defined as
\begin{equation*}
  \mathcal{N}_G(\Omega)=\left\{f|_\Omega:
  f\in\mathcal{N}_G(\mathbb{R}^d)~\textrm{and}~
  (f,g)_{\mathcal{N}_G(\mathbb{R}^d)}=0,
  \forall g\in\mathcal{N}_G(\mathbb{R}^d)
  ~\textrm{s.t.}~g|_\Omega=0\right\},
\end{equation*}
where $(f,g)_{\mathcal{N}_G(\mathbb{R}^d)}=\int_{\mathbb{R}^d}
\hat{f}(\omega)\overline{\hat{g}(\omega)}
e^{\frac{\|\omega\|_2^2}{4\delta^2}}\ud\omega$. For any $\Omega\subseteq\mathbb{R}^d$ and all $k\geqslant0$,
\begin{equation}\label{SRTF:GNSembedding}
  \mathcal{N}_G(\Omega)\subset W_2^k(\Omega)~~\textrm{with}~~
  \|f\|_{W_2^k(\Omega)}\leqslant C_G^kk^{k/2}\|f\|_{\mathcal{N}_G},
  ~~\forall f\in\mathcal{N}_G(\Omega),
\end{equation}
where $C_G=\sqrt{\max(\delta^{-d},1)(\frac{8\delta^2}{e}+2)}$ depends only on the shape parameter $\delta$ and the space dimension $d$, see Theorem $7.5$ of \cite{RiegerC2010_RBFSamplingInequalities} for details.

We can also consider inverse multiquadrics $M(x)=M_\delta(x)= (1/\delta^2+\|x\|_2)^{-\beta}$ for $\beta>\frac{d}{2}$, and the inner product of native spaces $\mathcal{N}_M(\mathbb{R}^d)$ can be defined as
\begin{equation*}
  (f,g)_{\mathcal{N}_M(\mathbb{R}^d)}=\int_{\mathbb{R}^d}
  \hat{f}(\omega)\overline{\hat{g}(\omega)}\widehat{M}^{-1}(\omega)\ud\omega,
  ~~\forall f,g\in\mathcal{N}_M(\mathbb{R}^d),
\end{equation*}
where $\widehat{M}(\omega)=\frac{2^{1-\beta}}{\Gamma(\beta)} (\delta\|\omega\|_2)^{\beta-d/2}K_{d/2-\beta}(\|\omega\|_2/\delta)$ and $K_v$ is the modified Bessel functions; and similarly to Gaussian kernels, for any $\Omega\subseteq\mathbb{R}^d$ and all $k\geqslant0$,
\begin{equation}\label{SRTF:MNSembedding}
  \mathcal{N}_M(\Omega)\subset W_2^k(\Omega)~~\textrm{with}~~
  \|f\|_{W_2^k(\Omega)}\leqslant C_M^kk^k\|f\|_{\mathcal{N}_M},
  ~~\forall f\in\mathcal{N}_M(\Omega),
\end{equation}
where $C_M>0$ depends only on $\beta$ and $d$, see Theorem $7.6$ of \cite{RiegerC2010_RBFSamplingInequalities} for details.

\section{Sparse residual tree}
\label{SRTF:s3}

Sparse residual tree is based on both tree decomposition and adaptive RBF explorations. Suppose $\epsilon_{\mathrm{E}}>0$ is the expected relative absolute error (RAE) for an approximation $s$ of the target function $f$ on the interpolation dataset $X$, where
\begin{equation}\label{SRTF:eq:RAE}
  \mathrm{RAE}=\frac{\max_i|s(x_i)-f(x_i)|}{\max_i|f(x_i)|}.
\end{equation}
For each child, for example, $X'\subset\Omega'$, which is $X\subset\Omega$ itself in the beginning, we need to (i) explore a sparse and proper RBF approximation $s_{r'}$ to minimize the $2$-norm of the current residual $r'(X')$; and then, (ii) split the dataset $X'$ into two proper subsets $X'_1$ and $X'_2$ as well as the domain $\Omega'$ into two proper subdomains $\Omega'_1$ and $\Omega'_2$, as shown in the following diagram. We call it an exploration-splitting process.
\begin{diagram}
&                      & X'\subset\Omega' &                      & \\
&      \ldTo(1,2)      &                  &      \rdTo(1,2)      & \\
& X'_1\subset\Omega'_1 &                  & X'_2\subset\Omega'_2 & \\
\end{diagram}

As mentioned above, the RBF exploration applies only to a sparse but sufficient subset of $X'$ and the centers of $s_{r'}$ are also sparse relative to the sparse subset. Hence, let us start with a sparsification of the dataset $X'$ when its number is large.

\subsection{Sparsification of datasets}
\label{SRTF:s3:1}

Except for updating the residual, we hope to improve efficiency by replacing $X'$ with its subset which has the same distribution of $X'$ if the number of $X'$ is large. Since $s_{r'}$ is only used to refine the relative global component of the current residual $r'$, it is not necessary to use all the data. Let $I$ be an index vector containing $N'_I(\leqslant N')$ unique integers selected randomly from $1$ to $N'$ inclusive, then $X'(I)$ is exactly what we need. Actually, from the independence of $X'$ and $I$, it follows that
\begin{equation*}
  \mathcal{P}_{X'(I)}(x)=\frac{\mathcal{P}_{X',I}(x,i)}{\mathcal{P}_{I}(i)}
  =\frac{\mathcal{P}_{X'}(x)\mathcal{P}_{I}(i)}{\mathcal{P}_{I}(i)}
  =\mathcal{P}_{X'}(x),~~x\in\Omega',~i\in I;
\end{equation*}
i.e., $X'(I)$ has the same probability distribution of $X'$. And the choice of the number $N'_I$ will be discussed later.

\subsection{Quasi-uniform subsequence}
\label{SRTF:s3:2}

Now we consider a method for generating a quasi-uniform subsequence of $X'(I)$, which is the basis for adaptive RBF explorations. To find a quasi-uniform subsequence $U'$ from $X'(I)=\{x'_{I(1)},\cdots,x'_{I(N'_I)}\}$, we start with the approximate mean point, that is,
\begin{equation}\label{SRTF:eq:Q1}
  u'_1=\arg\min_{x'\in X'(I)}\|x'-\overline{X'(I)}\|_2,~~\textrm{where}~~
  \overline{X'(I)}=\frac{1}{N'_I}\sum_{i=1}^{N'_I}x'_{I(i)}.
\end{equation}
And for known $U'_j=\{u'_1,\cdots,u'_j\}$, the subsequent point $u'_{j+1}$ is determined as
\begin{equation}\label{SRTF:eq:Qjp1}
  u'_{j+1}=\arg\max_{x'\in X'(I)}\left(\min_{1\leqslant l\leqslant j}
  \|x'-u'_l\|_2\right),
\end{equation}
i.e., $u'_{j+1}\in X'(I)$ is the point that maximizes the minimum of the set of distances from it to a point in $U'_j$. By storing an $N'_I$-dimensional distance vector and an $N'_I$-dimensional index vector, it only takes $\mathcal{O}(jN'_I)$ operations to generate $j$ quasi-uniform points $U'_j$ and determine the relationship between every point of $X'(I)$ and the Voronoi diagram of $U'_j$, see Fig. \ref{SRTF:fig:1} for examples.

\begin{figure}[!htb]
\centering
\begin{minipage}{1.0\textwidth}
  \includegraphics[width=0.49\textwidth]{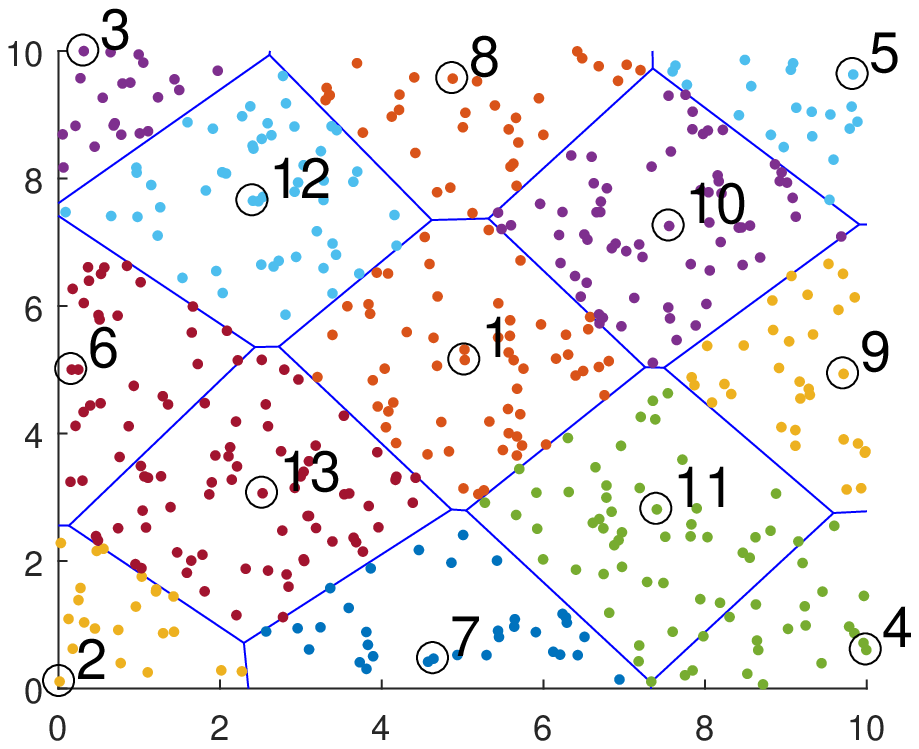}
  \includegraphics[width=0.49\textwidth]{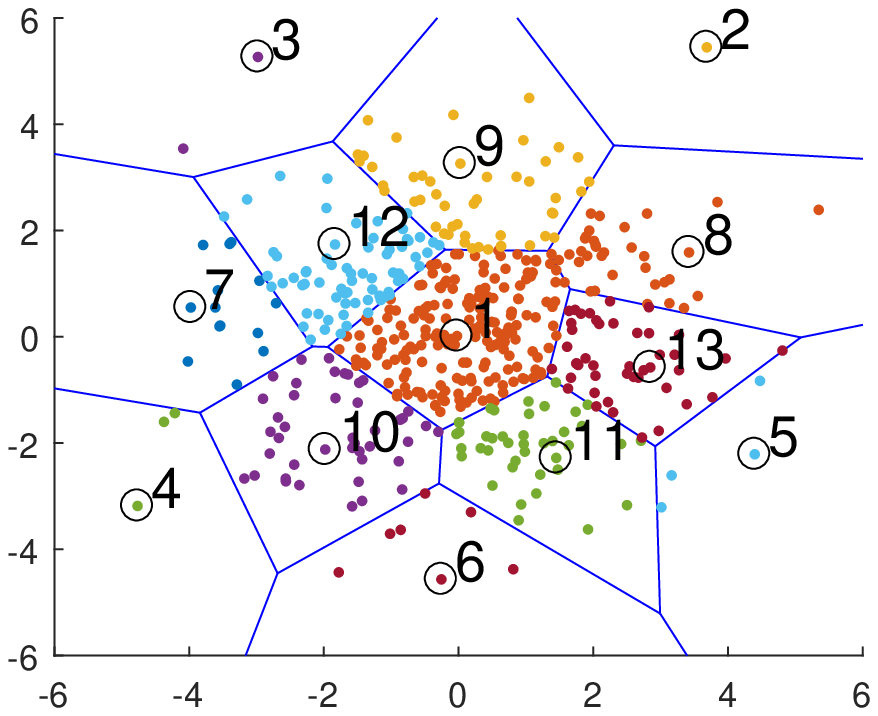}
\end{minipage}
\caption{\itshape\small{Left: $13$ quasi-uniform points $U_{13}$ ($\circ$) generated from $500$ uniform random samples ($\cdot$) and the Voronoi diagram of $U_{13}$. Right: $13$ quasi-uniform points $U_{13}$ ($ \circ$) generated from $500$ normal random samples ($\cdot$) with $\mu=(0,0)$ and $\Sigma = \textup{diag}([3,3])$ and the Voronoi diagram of $U_{13}$.}}
\label{SRTF:fig:1}
\end{figure}

\subsection{Adaptive RBF exploration}

The purpose of this adaptive exploration is to determine the centers of the RBF refinement $s_{r'}$ which is only used to refine the relative global component of $r'$. We first introduce the working parameters of SRT
\begin{equation}\label{SRTF:eq:WP}
  \omega=(\omega_1,\omega_2,\omega_3,\omega_4)^\mathrm{T},
\end{equation}
where $\omega_1=\kappa>0$ is the upper bound of condition numbers, $\omega_2>0$ is the termination error of explorations, $\omega_3\in(0,1)$ is the factor of shape parameters, and $\omega_4>0$ is the termination factor of tree nodes. For a fixed factor $\omega_3$, the current shape parameter can be determined as
\begin{equation*}
  \delta'=\sqrt{-\frac{\ln(\omega_3)}
  {\max_{x'\in X'(I)}\|x'-\overline{X'(I)}\|_2^2}},~~\textrm{where}~~
  \overline{X'(I)}=\frac{1}{N'_I}\sum_{i=1}^{N'_I}x'_{I(i)};
\end{equation*}
and the meaning of the remaining parameters will be clarified more clearly later.

Suppose $\{\chi'_1,\cdots,\chi'_{j'}\}$ are the centers inherited from its father, a reasonable idea is to choose the $(j'+1)$th center $\chi'_{j'+1}$ from the quasi-uniform subsequence $U'_{j'+d+1}$ which is generated by \eqref{SRTF:eq:Qjp1} with the initial $U'_{j'}=\{\chi'_1,\cdots,\chi'_{j'}\}$; for the root node, we choose $\chi'_1=u'_1\in U'_1$ given by \eqref{SRTF:eq:Q1}. Without loss of generality, for known $\{\chi'_1,\cdots,\chi'_j\}\subset U'_{j+d}$ with $j\geqslant j'$, we determine the $(j+1)$th center $\chi'_{j+1}$ from $U'_{j+d+1}-\{\chi'_1,\cdots,\chi'_j\}$ by the following procedure:
\begin{enumerate}
  \item From the recursive QR decomposition (as mentioned in section \ref{SRTF:s2}), $R_j$ and $Q_j^\mathrm{T}r'(X'(I))$ can be recursively obtained by $R_{j-1}$ and $Q_{j-1}^\mathrm{T}r'(X'(I))$ without computing $Q_{j-1}$, where
      \begin{equation*}
        Q_vR_v=\Phi_{X'(I),\{\chi'_1,\cdots,\chi'_v\}}\in\mathbb{R}^{N'_I,v},
        ~~1\leqslant v\leqslant j,
      \end{equation*}
      and $\Phi_{X'(I),\{\chi'_1,\cdots,\chi'_v\}}$ is generated by the Gaussian kernel $G_{\delta'}$.
  \item The temporary residual can be obtained by
      \begin{equation*}
        r'_j(X'(I))=r'(X'(I))-\sum_{1\leqslant i\leqslant j}
        \alpha_j^{(i)}e^{-\delta'^2\|X'(I)-\chi'_i\|_2^2}.
      \end{equation*}
      where the coefficients $\alpha_j=\left(\alpha^{(1)}_j,\cdots,
      \alpha^{(j)}_j\right)^\mathrm{T}=R_j^{-1}Q_j^\mathrm{T}r'(X'(I))$.
  \item Suppose $\{\Lambda_l\}_{l=1}^{j+d+1}$ be the Voronoi diagram of the set $U'_{j+d+1}$ and $\{\Lambda_l\}_{l\in\Gamma}$ are Voronoi regions with respect to those elements from the complementary set $U'_{j+d+1}-\{\chi'_1,\cdots,\chi'_j\}$, then
      \begin{equation}\label{SRTF:eq:centerjp1}
        \chi'_{j+1}=u'_{l^*}\in U'_{j+d+1}-\{\chi'_1,\cdots,\chi'_j\},
      \end{equation}
      where
      \begin{equation*}
        l^*=\arg\max_{l\in\Gamma}\sum_{x'\in\Lambda_l\cap X'(I)}
        \frac{|r'_j(x')|^2}{n_l},
      \end{equation*}
      and $n_l$ is the point number of $\Lambda_l\cap X'(I)$.
  \item And the termination criteria is
      \begin{equation}\label{SRTF:eq:centerT}
        \kappa(R)>\omega_1~~\textrm{or}~~
        \epsilon_j-\epsilon_{j+1}<\omega_2~~\textrm{or}~~j+1=N'_I,
      \end{equation}
      where $\kappa(R)=\frac{\max_l|R_{ll}|}{\min_l|R_{ll}|}$ is an estimation of the condition number $\|R^{-1}\|\|R\|$ and
      \begin{equation*}
        \epsilon_j=\sqrt{\frac{1}{N'_I}\sum_{x'\in X'(I)}(r'_j(x'))^2}.
      \end{equation*}
\end{enumerate}
To ensure that the centers is not too sparse, its number should usually be greater than $d+2$ (imagine a case that the domain $\Omega'$ is a $d$-dimensional simplex). Obviously, each newly selected center is in the Voronoi region with the largest mean squared error of the temporary residual. This allows the exploration to effectively capture the global component of the residual, see Fig. \ref{SRTF:fig:2} for examples.

\begin{figure}[!htb]
\centering
\begin{minipage}{1.0\textwidth}
  \includegraphics[width=0.49\textwidth]{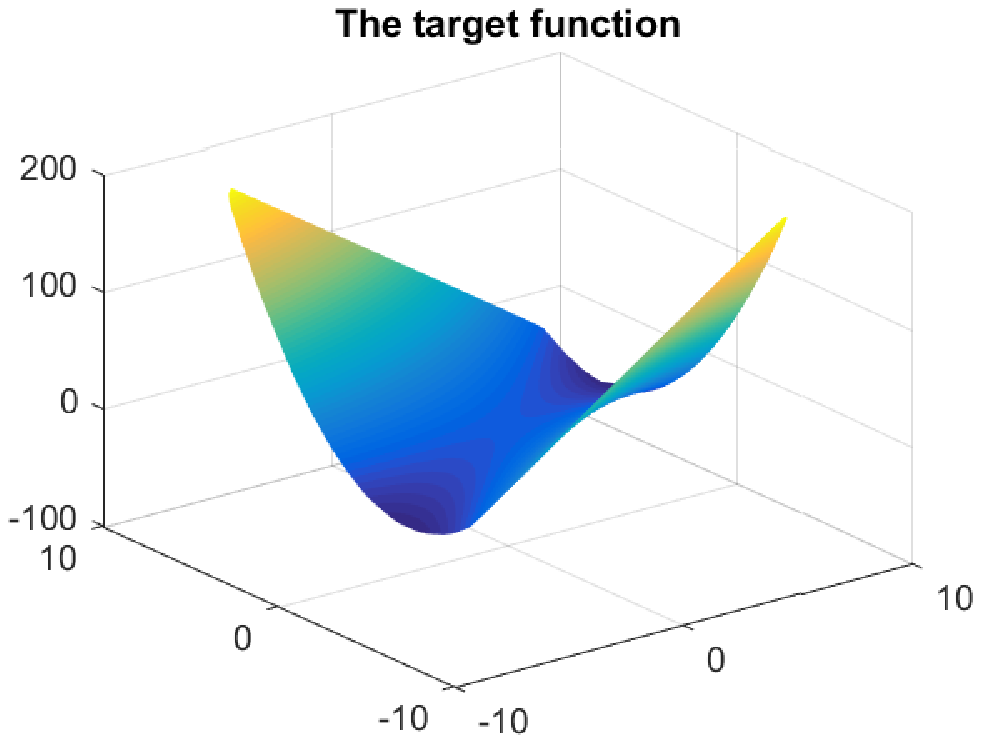}
  \includegraphics[width=0.49\textwidth]{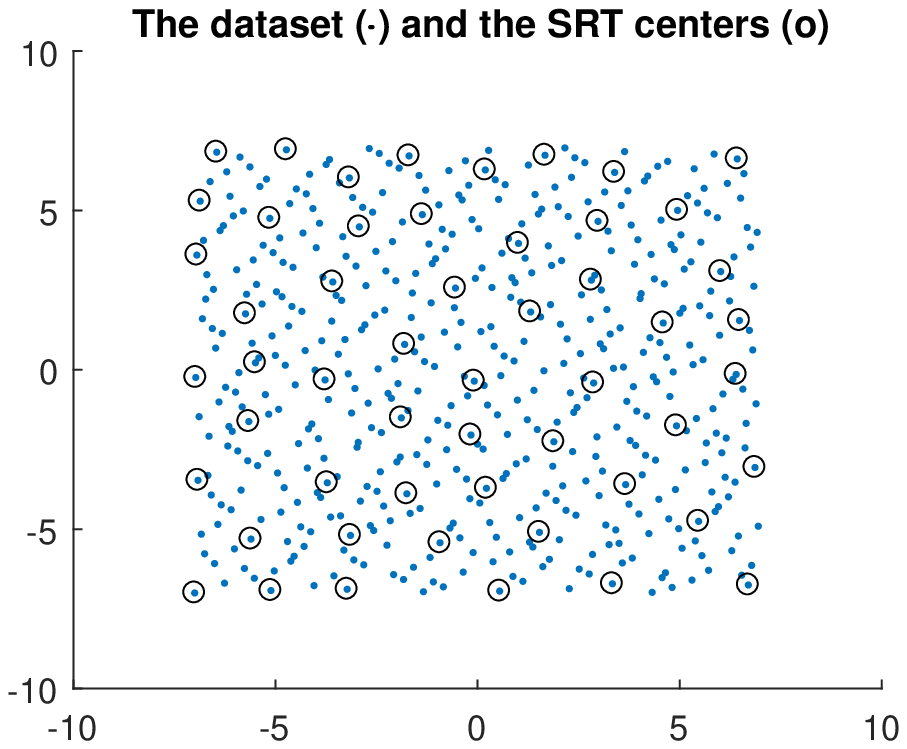}
\end{minipage}
\begin{minipage}{1.0\textwidth}
  \includegraphics[width=0.49\textwidth]{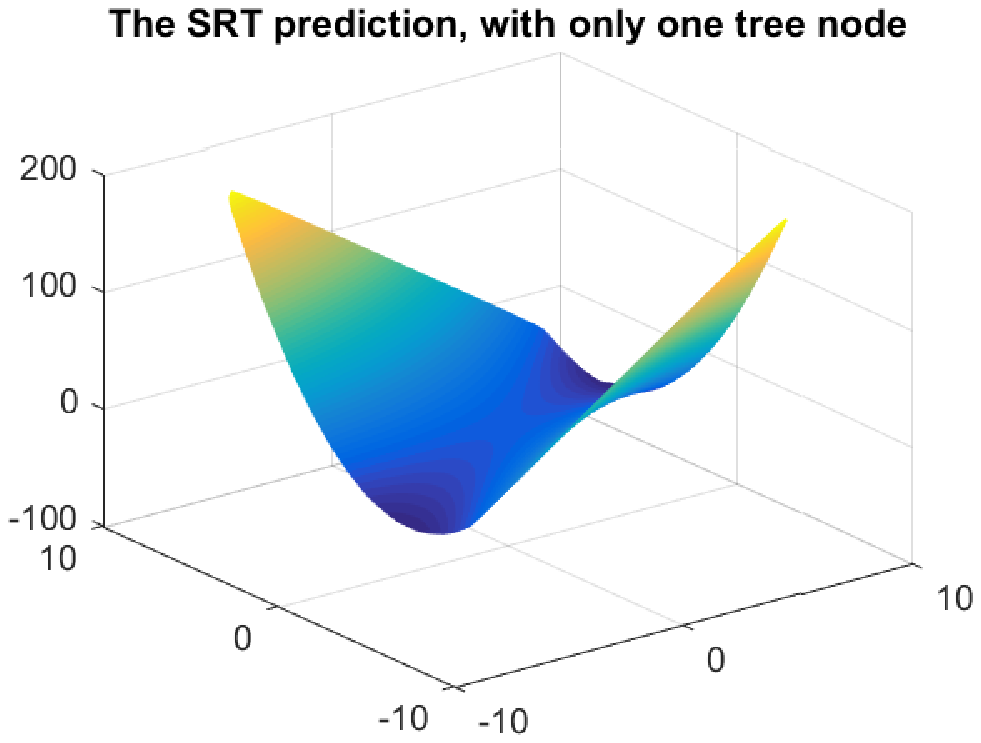}
  \includegraphics[width=0.49\textwidth]{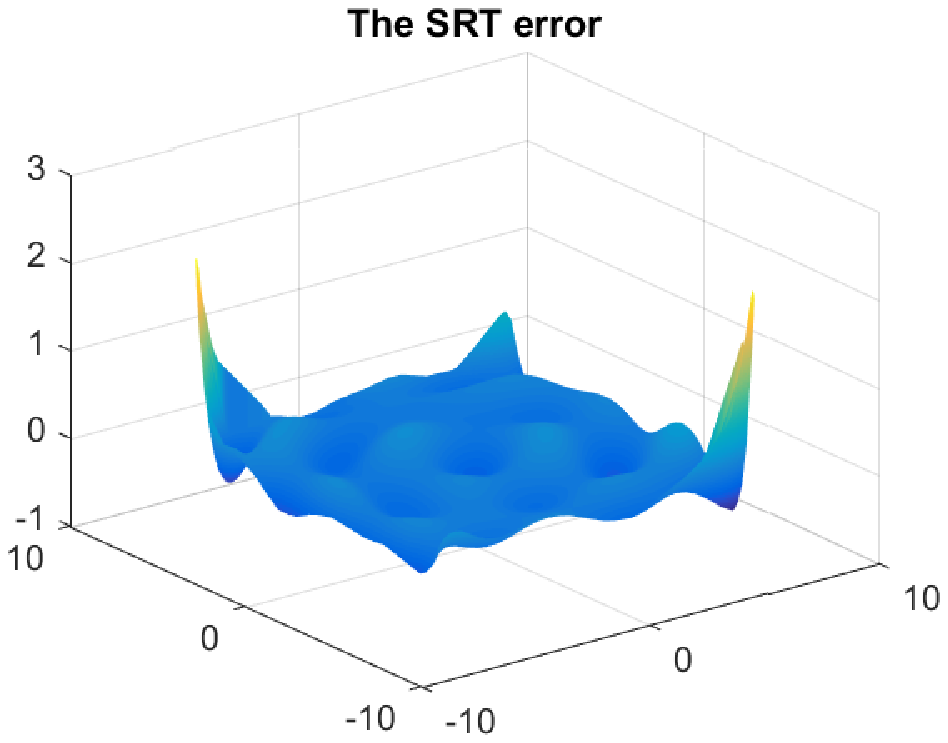}
\end{minipage}
\caption{\itshape\small{The target function $y=-2x_1x_2+2x_2^2,~x\in[-7,7]^2$, the dataset $X$ is a $2$-dimensional Halton sequence of length $500$, and the SRT prediction has only one node with $53$ centers for the expected RAE $\epsilon_{\mathrm{E}}=0.01$.}}
\label{SRTF:fig:2}
\end{figure}

The sparse RBF refinement $s_{r'}$ is obtained when the exploration is terminated, then we update the residual $r''$ on the full set $X'$. Let the final number of the centers is $N''$ and relevant coefficients $\alpha'=\left(\alpha'_1,\cdots,\alpha'_{N''}\right)^\mathrm{T}$, then
\begin{equation}\label{SRTF:eq:updError}
  r''(X')=r'(X')-\sum_{1\leqslant i\leqslant N''}
  \alpha'_ie^{-\delta'^2\|X'-\chi'_i\|_2^2}.
\end{equation}

In addition, assume that the number of all currently existing nodes is $M$ and $\{n_c^{(i)}\}_{i=1}^M$ is the set of the center number of each node, now define the average
\begin{equation}\label{SRTF:eq:Anc}
  \bar{n}_c=\frac{1}{M}\sum_{i=1}^Mn_c^{(i)},
\end{equation}
and we can use a certain multiple of the average $\bar{n}_c$, say $100$ times, as the value of $N'_I$ for the sparsification of the next node. For the initial node we usually take a fixed value related to the dimension $d$.

\subsection{Equal binary splitting and termination}
\label{SRTF:s3:4}

First we consider the selection of two splitting points, then use a hyperplane, whose normal is defined by these two points, to split all the points $X'$ into two parts as well as the domain $\Omega'$ into two subdomains. Clealy, since the half space and $\Omega'$ are both convex, each subdomain is also convex. In order to block the spread of error, we expect to separate the points with large errors from those with small errors. First, we generate $d+1$ quasi-uniform points $U'_{d+1}$ of $X'(I)$ by the method of subsection \ref{SRTF:s3:2} with a different starting point:
\begin{equation*}
  u'_1=\arg\max_{x'\in X'(I)}\|x'-\overline{X'(I)}\|_2,~~\textrm{where}~~
  \overline{X'(I)}=\frac{1}{N'_I}\sum_{i=1}^{N'_I}x'_{I(i)}.
\end{equation*}
Assume that the domain $\Omega'$ is a $d$-dimensional simplex and $X'$ is dense enough, then $U'_{d+1}$ can almost be viewed as its vertices. Let $\{\Lambda_l\}_{l=1}^{d+1}$ be the Voronoi diagram of $U'_{d+1}$, then the first splitting point is determined as
\begin{equation*}
  x'_a=u'_{l^*},
\end{equation*}
where $l^*=\arg\max_l\sum_{x'\in\Lambda_l\cap X'(I)}\frac{|r''(x')|^2}{n_l}$ and $n_l$ is the point number of $\Lambda_l\cap X'(I)$. Then the second splitting point is determined as
\begin{equation*}
  x'_b=\arg\max_{x'\in X'(I)}\|x'-x'_a\|_2.
\end{equation*}
Then, according to the projections of $X'$ in the direction $x'_b-x'_a$ and its median, $X'$ can be splitted into $X'_1$ and $X'_2$ with the sizes $\lceil\frac{N'}{2}\rceil$ and $N'-\lceil\frac{N'}{2}\rceil$, respectively; where $\lceil t\rceil$ denotes the least integer greater than or equal to $t$. Specifically, let $\vec{n}'=(x'_b-x'_a)^\mathrm{T}$, then the projections
\begin{equation*}
  P_{\vec{n}'}(X')=X'\vec{n}',~~\textrm{where}~~
  X'\in\mathbb{R}^{N'\times d}~\textrm{and}~\vec{n}'\in\mathbb{R}^{d\times 1};
\end{equation*}
let $c'=\textrm{median}(P_{\vec{n}'}(X'))$, then $X'_1$ and $X'_2$ can be given as
\begin{equation}\label{SRTF:eq:split1}
  X'_1=\{x'\in X':P_{\vec{n}'}(x')\leqslant c'\}~~\textrm{and}~~X'_2=X'-X'_1;
\end{equation}
and similarly, $\Omega'_1$ and $\Omega'_2$ can be given as
\begin{equation}\label{SRTF:eq:split2}
  \Omega'_1=\{x'\in \Omega':x'\vec{n}'\leqslant c'\}
  ~~\textrm{and}~~\Omega'_2=\Omega'-\Omega'_1.
\end{equation}
Since the local high-frequency error tends to propagate over the entire domain, blocking its propagation is very important for a sparse approximation, and this is the motivation for designing the above splitting, see Fig. \ref{SRTF:fig:3}.
\begin{figure}[!htb]
\centering
\begin{minipage}{1.0\textwidth}
  \includegraphics[width=0.49\textwidth]{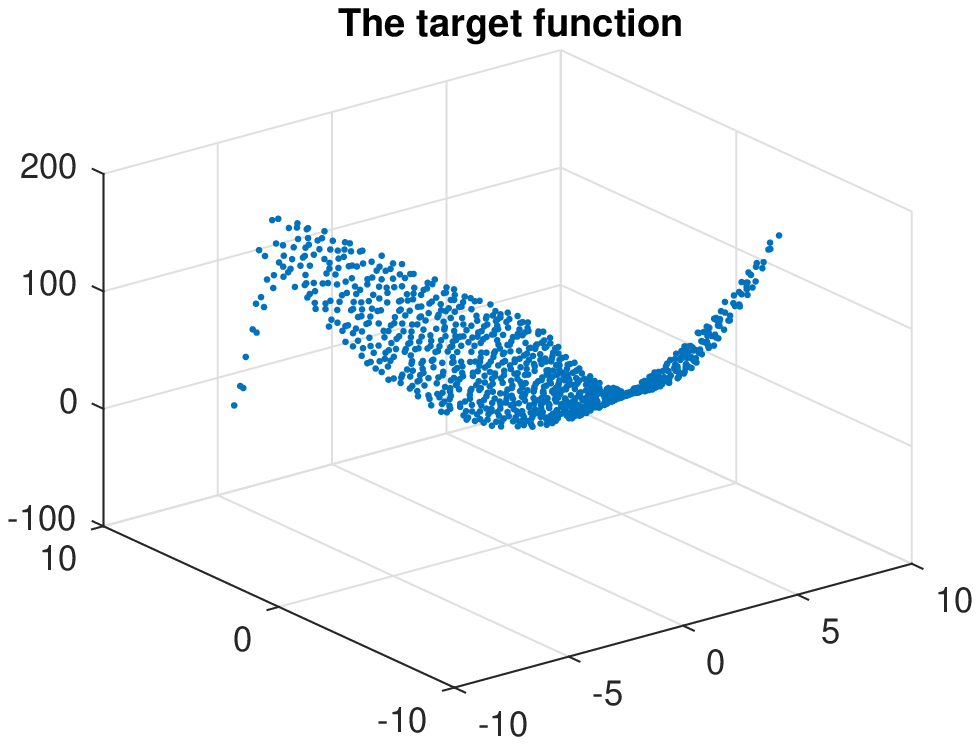}
  \includegraphics[width=0.49\textwidth]{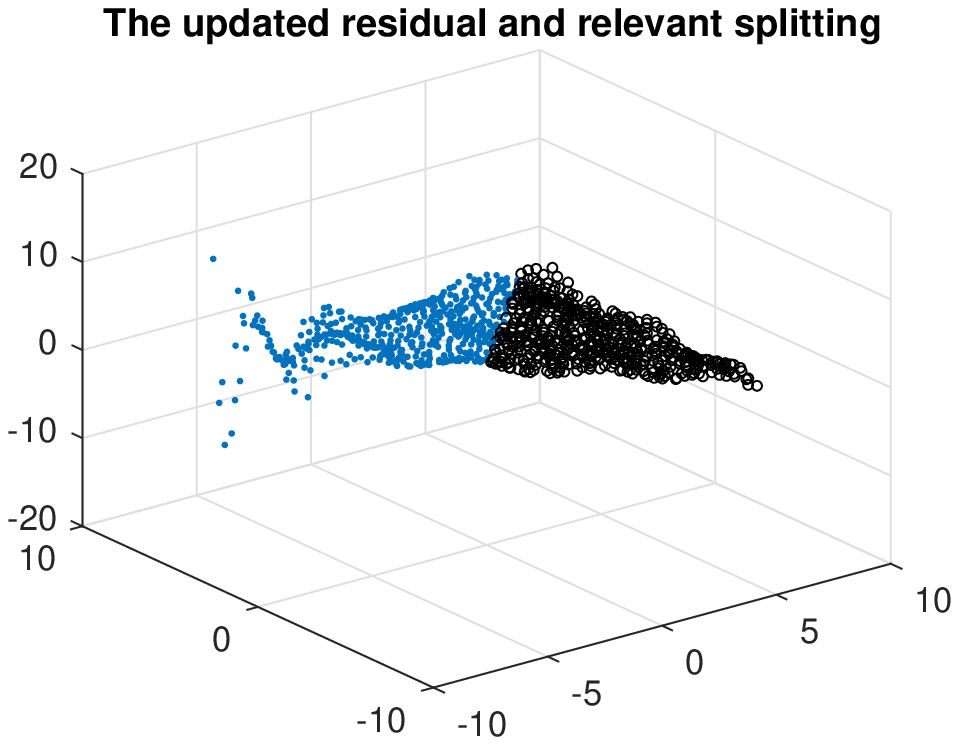}
\end{minipage}
\caption{\itshape\small{The target function $y=-2x_1x_2+2x_2^2-200\exp(-0.7(x_1+7)^2-0.7(x_2-7)^2),~x\in\{x\in[-7,7]^2:x_1+x_2 >0\}$, the dataset $X$ is a $2$-dimensional Halton sequence of length $1000$.}}
\label{SRTF:fig:3}
\end{figure}

This exploration-splitting process finally stops if the expected RAE $\epsilon_{\mathrm{E}}$ is reached or the data is insufficient at the current tree node. Another important use of the average $\bar{n}_c$ defined in \eqref{SRTF:eq:Anc} is to determine whether the data is sufficient. Obviously, a sparse approximation must be based on relatively sufficient data, so if the size of $X'_1$ or $X'_2$ is less than $\omega_4$ times the average $\bar{n}_c$ and the RAE of residual still does not reach the expected $\epsilon_{\mathrm{E}}$, then we consider that the relevant node is lack of data, terminate further operations and record the node. A proper $\omega_4$ can guarantee that the prediction does not over-fit the data.

\subsection{SRT prediction and its error characteristics}

Suppose $s'$ is the current approximation on the domain $\Omega'$ and $s'_{r''(X'_i)}$ is the refinement on $\Omega'_i~(i=1,2)$. Then the next approximation $s'_i$ on $\Omega'_i~(i=1,2)$ can be given as
\begin{equation}\label{SRTF:eq:J}
  s'_i(x)=s'(x)+s'_{r''(X'_i)}(x),~~\forall x\in\Omega'_i.
\end{equation}
It is clear that the SRT prediction is actually piecewise smooth on the original domain $\Omega$, hence the error of each piece will be significantly larger near the boundary.

The following example illustrates the error characteristics of SRTs. Although the SRT prediction, as shown on the left-hand side of Fig. \ref{SRTF:fig:4}, can adaptively build a piecewise and sparse approximation according to local features of the target function, the approximation error, as shown on the right-hand side of Fig. \ref{SRTF:fig:4}, may be significantly larger near the boundary of each piece. Hence, we will introduce the sparse residual forest for overcoming this boundary effect of the error in the next section.

\begin{figure}[!htb]
\centering
\begin{minipage}{1.0\textwidth}
  \includegraphics[width=0.49\textwidth]{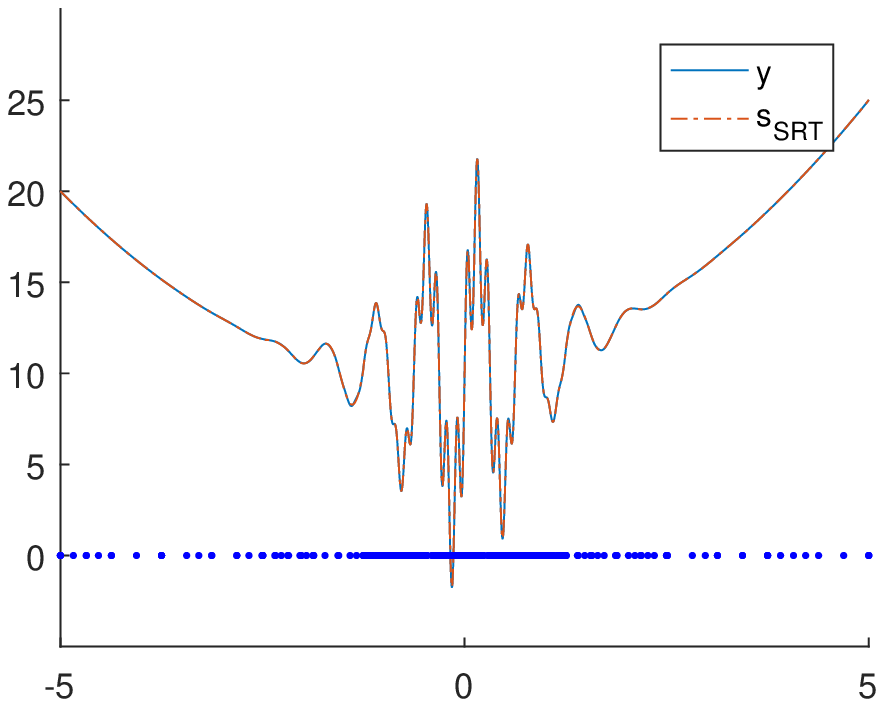}
  \includegraphics[width=0.49\textwidth]{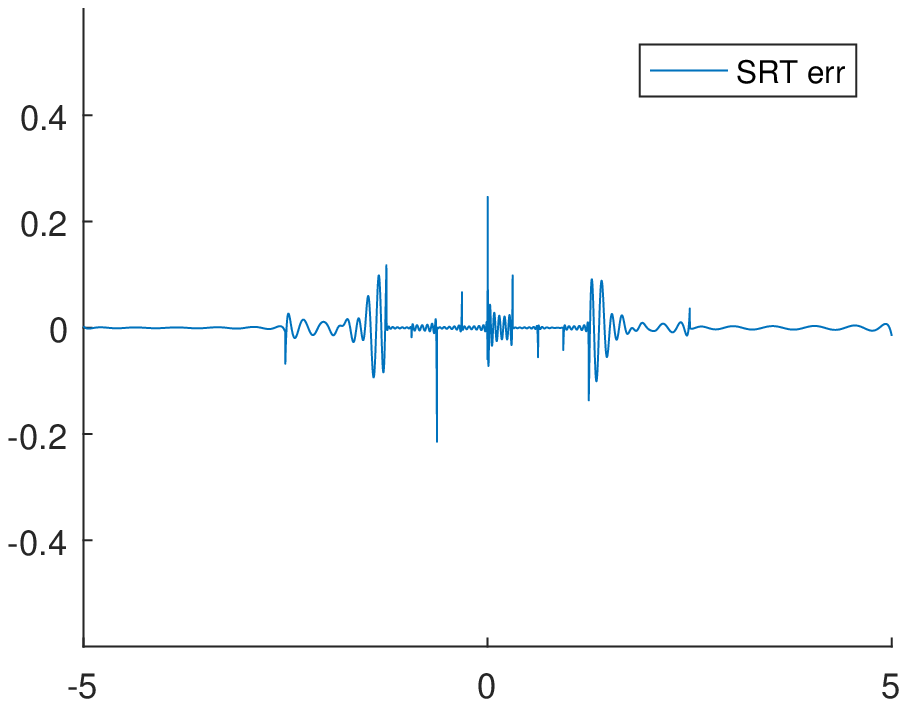}
\end{minipage}
\caption{\itshape\small{The target function $y=10+\frac{x}{2}+\frac{x^2}{2}+8e^{-\frac{7x^2}{10}} \sin(10x)+4e^{-2x^2}\sin(50x),~x\in[-5,5]$, the equally spaced dataset $X=\{\frac{10i}{999}-5\}_{i=0}^{999}$. Left: the original function $y$, the SRT prediction $s_{\mathrm{SRT}}$ with the expected RAE $\epsilon_{\mathrm{E}}=0.01$, and all the SRT centers ($\cdot$). Right: the error $y-s_{\mathrm{SRT}}$.}}
\label{SRTF:fig:4}
\end{figure}

The partition of unity is also one of the methods to address this issue. By introducing appropriate overlapping domains and rapidly decaying weight functions, the boundary effect of the error can be alleviated to some extent. However, since the overlapping domains usually cannot be too small and the depth of the tree is often not small, its time and space costs are significantly higher than $\mathcal{O}(N\log_2N)$. Instead, sparse residual forests still have the same cost as SRTs. And it provides even better performance than the partition of unity based method in terms of accuracy.

\section{Sparse residual forest}
\label{SRTF:s4}

Sparse residual forest (SRF) is a combination of SRT predictors with different tree decompositions. It provides an opportunity to avoid those predictions near the boundary and then use the average value of the remaining predictions to enhance both stability and convergence. First, we introduce a random splitting for SRTs. It can help generate random tree decompositions.

\subsection{Random binary splitting}
\label{SRTF:s4:1}

To get a random splitting, we only need to replace the median with a random percentile in \eqref{SRTF:eq:split1}. Let $p_r$ be a randomly selected integer from $37$ to $62$ inclusive, then $c'$ can be redefined as
\begin{equation*}
  c'=\textrm{percentile}(P_{\vec{n}'}(X'),p_r),
\end{equation*}
where $\textrm{percentile}(Z,p_r)$ denotes the percentile of the values in a data vector $Z$ for the percentage $p_r$. Note that $0.618$ is the golden ratio and this method depends on the values of a random vector sampled independently and with the same distribution.

\subsection{SRF prediction}

Suppose $n_t$ is the number of SRTs in the SRF, we usually apply the equal splitting to generate the first SRT and the random splitting to create the remaining $n_t-1$ SRTs. SRF helps us to avoid those predictions with large squared deviations and to use the average value of the remaining predictions to enhance both stability and convergence.

For any $x\in\Omega$, let $s_{\textrm{SRT}}^{(i)}(x)$ be the $i$th SRT prediction ($1\leqslant i\leqslant n_t$), then the squared deviation
\begin{equation*}
  \sigma_i^2(x)=\left(s_{\textrm{SRT}}^{(i)}(x)-
  \frac{1}{n_t}\sum_{j=1}^{n_t}s_{\textrm{SRT}}^{(j)}(x)\right)^2,
\end{equation*}
further, let the indicator set
\begin{equation*}
  I_F=\left\{1\leqslant i\leqslant n_t:\sigma_i^2(x)
  <\frac{1}{n_t}\sum_{j=1}^{n_t}\sigma_j^2(x)\right\},
\end{equation*}
then the SRF prediction
\begin{equation}\label{SRTF:eq:SRF}
  s_{\textrm{SRF}}(x)=\frac{1}{n_{I_F}}\sum_{i\in I_F}s_{\textrm{SRT}}^{(i)}(x),
  ~~\textrm{where}~n_{I_F}~\textrm{is the size of}~I_F.
\end{equation}
The indicator set $I_F$ here is used to avoid those predictions near the boundaries. In practice, as shown in Fig. \ref{SRTF:fig:5}, SRFs composed of a small number of SRTs perform quite well than individual SRTs; and in theory, similar to random forests \cite{BreimanL2001A_RandomForest}, the error for SRFs converges with probability $1$ to a limit as $n_t$ becomes large, see Fig. \ref{SRTF:fig:6} for examples and subsection \ref{SRTF:s5:3} for details.

Although SRF predictions usually have smaller errors when the SRT number $n_t$ is larger, we usually do not recommend choosing a large $n_t$, which means $n_t$ times the storage and computational cost.

\begin{figure}[!htb]
\centering
\begin{minipage}{1.0\textwidth}
  \includegraphics[width=0.49\textwidth]{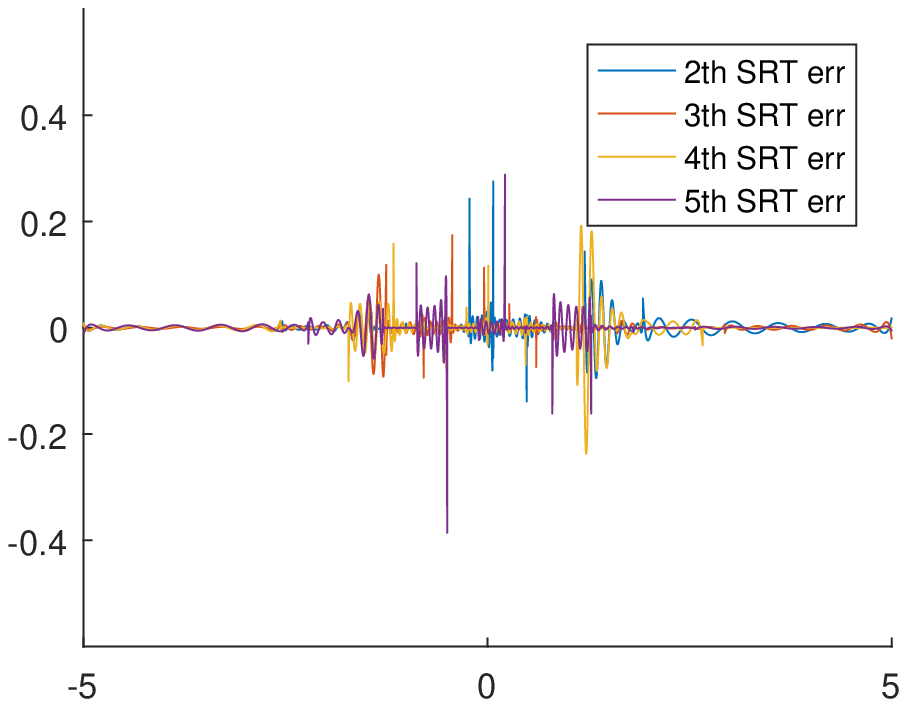}
  \includegraphics[width=0.49\textwidth]{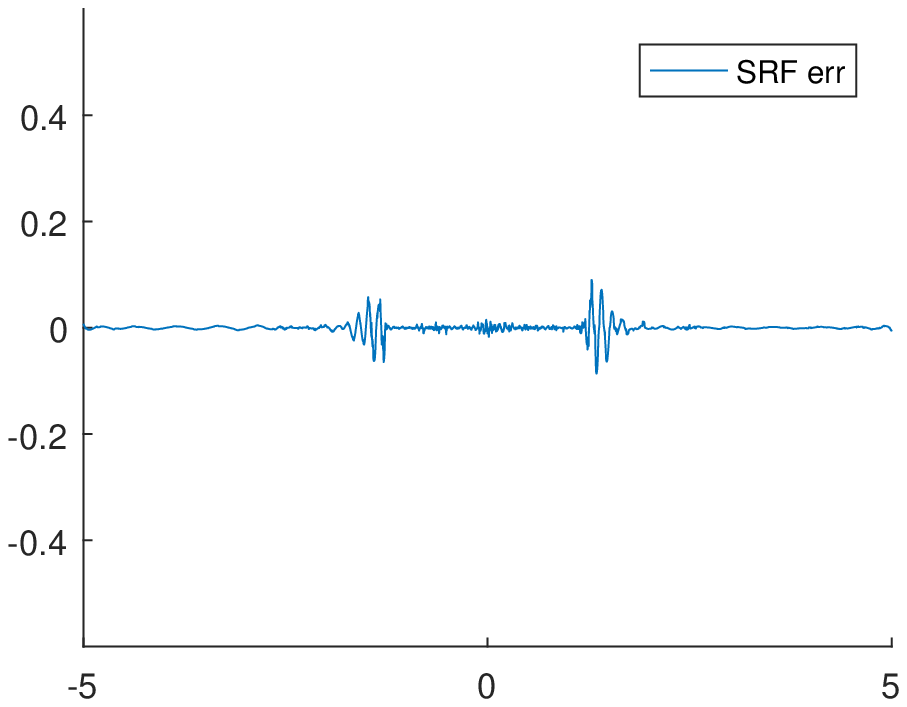}
\end{minipage}
\caption{\itshape\small{A SRF of $5$ SRTs with the expected RAE $\epsilon_{\mathrm{E}}=0.01$ for the example in Fig. \ref{SRTF:fig:4}. Left: the errors of the remaining $4$ SRTs. Right: the error of the SRF prediction.}}
\label{SRTF:fig:5}
\end{figure}
\begin{figure}[!htb]
\centering
\begin{minipage}{1.0\textwidth}
  \includegraphics[width=0.49\textwidth]{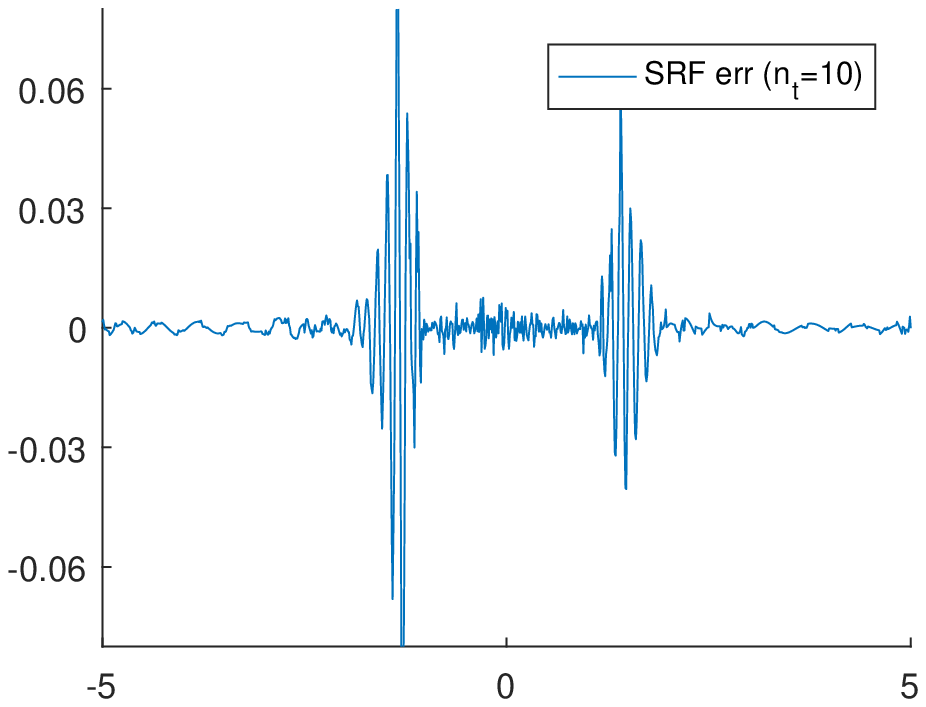}
  \includegraphics[width=0.49\textwidth]{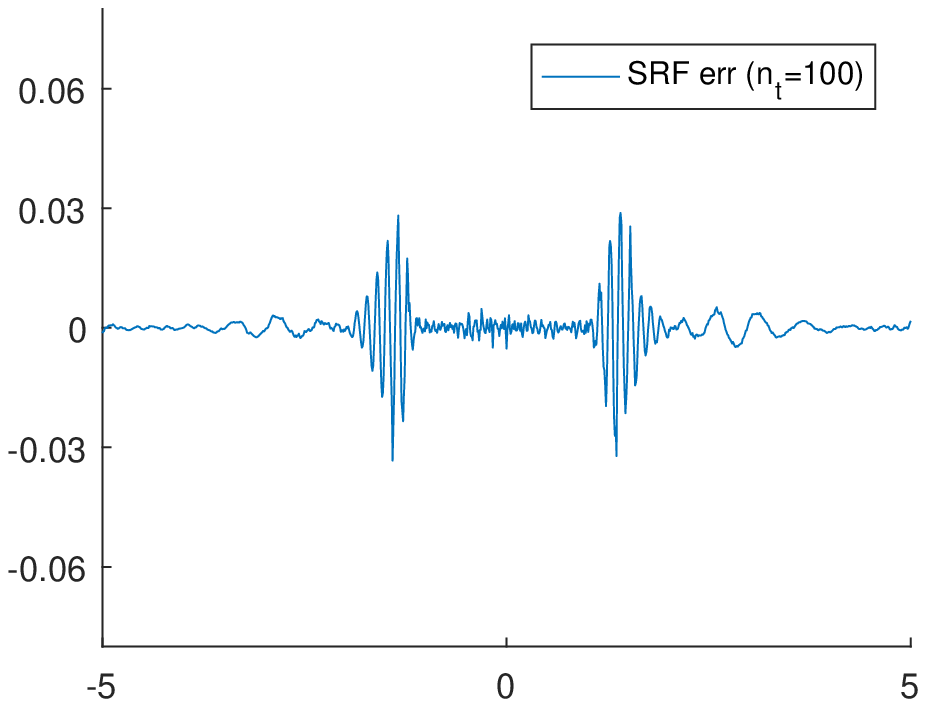}
\end{minipage}
\caption{\itshape\small{The errors of SRF predictions with different values of $n_t$ and the same expected RAE $\epsilon_{\mathrm{E}}=0.01$ for the example in Fig. \ref{SRTF:fig:4}.}}
\label{SRTF:fig:6}
\end{figure}

\section{Theory}
\label{SRTF:s5}

\subsection{Stability properties}

Suppose $\Omega_{L-1}$ is a leaf node, that is at the lowest level in a SRT, and $L$ levels of approximation, then there exists a domain sequences $\Omega_0\supset\Omega_1\supset\cdots \supset\Omega_{L-1}$ and a relevant dataset sequences $X_0\supset X_1\supset\cdots\supset X_{L-1}$ with relevant sizes $N_0>N_1>\cdots>N_{L-1}$ and shape parameters $\delta_0<\delta_1<\cdots<\delta_{L-1}$, where $\Omega_0=\Omega$ is convex, $X_0=X$ and $N_0=N$; and then, the SRT prediction of the target function $f$ is
\begin{equation}\label{SRTF:eq:SRT}
  s_{\mathrm{SFT}}(x)=\sum_{i=0}^{L-1}s_i(x),~~\forall x\in\Omega_{L-1},
\end{equation}
and the final residual
\begin{equation}\label{SRTF:eq:RSRT}
  r_L(x)=f(x)-s_{\mathrm{SFT}}(x),~~\forall x\in\Omega_{L-1},
\end{equation}
where $s_i(x)=\sum_{j=1}^{N'_i}\alpha_i^{(j)}G_{\delta_i}(x-\chi_i^{(j)})\in \mathcal{N}_{G_{\delta_i}}(\Omega_i)$ is the LS approximation of the residual $r_i(X_i)$ with respect to the centers $\chi_i=\{\chi_i^{(j)}\}_{j=1}^{N'_i}\in X_i$, and $r_{i+1}=r_i-s_i$ with $r_0=f$. Then, for any $1\leqslant i\leqslant L-1$, it follows that
\begin{equation*}
  \big(s_i(X_i),r_{i+1}(X_i)\big)_{\ell_2}=0~~\textrm{and}~~
  \alpha_i=R_i^{-1}Q_i^\mathrm{T}r_i(X_i)=R_i^{-1}Q_i^\mathrm{T}s_i(X_i),
\end{equation*}
where $Q_iR_i$ is the QR decomposition of the current matrix $\Phi_i=\Phi_{X_i,\chi_i}$ generated by the kernel $G_{\delta_i}$. If $\tau_i$ is the smallest singular value of $R_i$, then
\begin{equation}\label{SRTF:eq:Ti}
  \|\alpha_i\|_2\leqslant\tau_i^{-1}\|s_i(X_i)\|_2.
\end{equation}
According to the orthogonality of $s_i(X_i)$ and $r_{i+1}(X_i)$, we can obtain the following recurrence relations
\begin{equation*}
  \|r_i(X_i)\|_2^2=\|r_{i+1}(X_i)\|_2^2+\|s_i(X_i)\|_2^2,
  ~~0\leqslant i\leqslant L-1,
\end{equation*}
and
\begin{equation*}
  \|r_i(X_{i-1})\|_2^2>\|r_i(X_i)\|_2^2,~~1\leqslant i\leqslant L,
\end{equation*}
thus, it follows that
\begin{equation*}
  \|f_X\|_2^2=\|s_0(X_0)\|_2^2+\|r_1(X_0)\|_2^2
  >\sum_{i=0}^{L-1}\|s_i(X_i)\|_2^2+\|r_L(X_{L-1})\|_2^2.
\end{equation*}
Together with \eqref{SRTF:eq:Ti}, we proved the following theorem.
\begin{thm}\label{SRTF:thm:Cbound}
Suppose $s_\mathrm{SFT}$ is a SRT prediction of a function $f$ on a leaf node $\Omega_{L-1}\subset\Omega$ with respect to the data $(X,f_X)$, as defined in \eqref{SRTF:eq:SRT}. Let $\alpha_i$ be the coefficients of the $i$th level LS approximation $s_i$, then
\begin{equation*}
  \sum_{i=0}^{L-1}\|\alpha_i\|_2\leqslant\tau^{-1}\cdot\|f_X\|_2,
\end{equation*}
where $\tau=\min_{1\leqslant i\leqslant L-1}\tau_i$ and the constants $\tau_i$ comes from \eqref{SRTF:eq:Ti}.
\end{thm}

Note that this theorem obviously holds for our SRTs with sparsification processes introduced in subsection \ref{SRTF:s3:1}. And now we can prove the following theorem.
\begin{thm}\label{SRTF:thm:WNbound}
Under the supposition of Theorem \ref{SRTF:thm:Cbound}. For all $1\leqslant p<\infty$, $k\in\mathbb{N}_0$, $\delta\geqslant\delta_{L-1}$, and any leaf node $\Omega_{L-1}$ of the prediction $s_{\mathrm{SRT}}$, it holds that
\begin{equation*}
  |s_\mathrm{SFT}|_{W_p^k(\Omega_{L-1})}\leqslant C_W\cdot\tau^{-1}\cdot\|f_X\|_2
  ~~\textrm{and}~~\|s_\mathrm{SFT}\|_{\mathcal{N}_{G_\delta}(\Omega_{L-1})}
  \leqslant C_{\mathcal{N}}\cdot\tau^{-1}\cdot\|f_X\|_2,
\end{equation*}
where the constant $\tau$ comes from Theorem \ref{SRTF:thm:Cbound}, the constant $C_W$ depends only on $\delta_0,\delta_{L-1},d,p$ and $k$, and the constant $C_{\mathcal{N}}$ depends only on $\delta_{L-1}$ and $d$.
\end{thm}
\begin{proof}
To prove the first inequality, observe that
\begin{align*}
  |s_i|_{W_p^k(\Omega_{L-1})}\leqslant\!\left(\sum_{|r|=k}\sum_j
  \big|\alpha_i^{(j)}\big|^p\big\|D^rG_{\delta_i}\big\|^p_{L_p(\Omega_{L-1})}
  \right)^{1/p}\!\!\leqslant M_{\delta_i}^{p,k}\|\alpha_i\|_p\leqslant C_1M_{\delta_i}^{p,k}\|\alpha_i\|_2,
\end{align*}
where $M_{\delta_i}^{p,k}=\left(\sum_{|r|=k}
\|D^rG_{\delta_i}\|^p_{L_p(\mathbb{R}^d)}\right)^{1/p}$, and for any $0\leqslant i\leqslant L-1$, $M_{\delta_i}^{p,k}<M_{\delta_{L-1}}^{p,k}$ when $k>1$; or $M_{\delta_i}^{p,k}<M_{\delta_0}^{p,k}$ when $k<1$; or $M_{\delta_i}^{p,k}=M^{p,k}$ is independent of $\delta_i$ when $k=1$. Together with Theorem \ref{SRTF:thm:Cbound}, we have
\begin{align*}
  |s_\mathrm{SFT}|_{W_p^k(\Omega_{L-1})}\leqslant&\sum_{i=0}^{L-1}
  |s_i|_{W_p^k(\Omega_{L-1})}\leqslant C_W\cdot\tau^{-1}\cdot\|f_X\|_2,
\end{align*}
where $C_W=C_1M_{\delta_{L-1}}^{p,k}$ when $k>1$, or $C_W=C_1M_{\delta_0}^{p,k}$ when $k<1$, or $C_W=C_1M^{p,k}$ when $k=1$.

To prove the second inequality, observe that for any $s_i\in \mathcal{N}_{G_{\delta_i}}(\Omega_i)$, there is a natural extension $\mathcal{E}s_i\in\mathcal{N}_{G_{\delta_i}}(\mathbb{R}^d)$ with $\|\mathcal{E}s_i\|_{\mathcal{N}_{G_{\delta_i}}(\mathbb{R}^d)}=
\|s_i\|_{\mathcal{N}_{G_{\delta_i}}(\Omega_i)}$. From the definition of native spaces of Gaussians, we see that $\mathcal{E}s_i\in\mathcal{N}_{G_\delta} (\mathbb{R}^d)$ with
\begin{equation}\label{SRTF:eq:keyNG}
  \|\mathcal{E}s_i\|_{\mathcal{N}_{G_\delta}(\mathbb{R}^d)}\leqslant
  \|\mathcal{E}s_i\|_{\mathcal{N}_{G_{\delta_i}}(\mathbb{R}^d)},
\end{equation}
where $\delta\geqslant\delta_{L-1}>\cdots>\delta_0$; and further, the restriction $\mathcal{E}s_i|\Omega_{L-1}=s_i|\Omega_{L-1}$ of $\mathcal{E}s_i$ to $\Omega_{L-1}\subseteq\Omega_i$ is contained in $\mathcal{N}_{G_\delta}(\Omega_{L-1})$ with
\begin{equation*}
  \|s_i|\Omega_{L-1}\|_{\mathcal{N}_{G_\delta}(\Omega_{L-1})}\leqslant
  \|\mathcal{E}s_i\|_{\mathcal{N}_{G_\delta}(\mathbb{R}^d)},
\end{equation*}
hence, we have $\|s_i|\Omega_{L-1}\|_{\mathcal{N}_{G_\delta}(\Omega_{L-1})} \leqslant\|\mathcal{E}s_i\|_{\mathcal{N}_{G_{\delta_i}}(\mathbb{R}^d)}$, and then
\begin{equation*}
  \|s_\mathrm{SFT}\|_{\mathcal{N}_{G_\delta}(\Omega_{L-1})}
  \leqslant\sum_{i=1}^{L-1}
  \|s_i|\Omega_{L-1}\|_{\mathcal{N}_{G_\delta}(\Omega_{L-1})}
  \leqslant\sum_{i=1}^{L-1}
  \|\mathcal{E}s_i\|_{\mathcal{N}_{G_{\delta_i}}(\mathbb{R}^d)}.
\end{equation*}

Together with Theorem \ref{SRTF:thm:Cbound} and
\begin{align*}
  \|\mathcal{E}s_i\|^2_{\mathcal{N}_{G_{\delta_i}}(\mathbb{R}^d)}
  \!=\!\int_{\mathbb{R}^d}|\hat{s}_i(\omega)|^2
  e^{\frac{\|\omega\|_2^2}{4\delta_i^2}}\ud\omega
  \leqslant\|\alpha_i\|_1^2\!
  \int_{\mathbb{R}^d}e^{-\frac{\|\omega\|_2^2}{4\delta_i^2}}\ud\omega
  \leqslant C_2^2(2\delta_{L-1})^d\pi^{d/2}\|\alpha_i\|_2^2
\end{align*}
we finally have $\|s_\mathrm{SFT}\|_{\mathcal{N}_{G_\delta}(\Omega_{L-1})}
<C_{\mathcal{N}}\cdot\tau^{-1}\cdot\|f_X\|_2$, where $C_{\mathcal{N}}=C_2(2\delta_{L-1})^{d/2}\pi^{d/4}$.
\end{proof}
\begin{rem}
See Theorems $10.46$ and $10.47$ in \cite{WendlandH2005B_ScatteredData} for details about the restriction and extension of functions from certain native spaces.
\end{rem}
\begin{rem}\label{SRTF:rem:M}
The second inequality depends on the embeddings \eqref{SRTF:eq:keyNG} of native spaces of Gaussians. As mentioned in section \ref{SRTF:s2}, the Fourier transform of the inverse multiquadrics is $\widehat{M}_\delta(\omega)=\frac{2^{1-\beta}}{\Gamma(\beta)} (\delta\|\omega\|_2)^{\beta-d/2}K_{d/2-\beta}(\|\omega\|_2/\delta)$, then for any $\beta>\frac{d}{2}$ and $\delta\geqslant\delta_i$, $\widehat{M}_\delta^{-1}(\omega) \leqslant\widehat{M}_{\delta_i}^{-1}(\omega)$, and then, for an inverse multiquadric based $s_i$,
\begin{equation}\label{SRTF:eq:keyNM}
  \|\mathcal{E}s_i\|_{\mathcal{N}_{M_\delta}(\mathbb{R}^d)}\leqslant
  \|\mathcal{E}s_i\|_{\mathcal{N}_{M_{\delta_i}}(\mathbb{R}^d)},
\end{equation}
hence, the second inequality also holds for native spaces of inverse multiquadrics.
\end{rem}

\subsection{Error estimates for SRTs}

\begin{thm}\label{SRTF:thm:EEW}
Under the supposition of Theorem \ref{SRTF:thm:Cbound}. If $f\in W_p^k(\Omega)$ and $r_L$ is the residual $f-s_{\mathrm{SRT}}$ on an arbitrary leaf node $\Omega_{L-1}$, then for any $1\!\leqslant\!q\!\leqslant\!\infty$, $\gamma\in\mathbb{N}_0^d$, and $1\!\leqslant\!p\!<\!\infty$ with $k>|\gamma|+d/p$ if $p>1$, or with $k\geqslant|\gamma|+d$ if $p=1$, it holds that
\begin{equation*}
  \|D^\gamma r_L\|_{L_q(\Omega_{L-1})}\!\leqslant\!
  C\left[h^{k-|\gamma|-\left(\frac{d}{p}-\frac{d}{q}\right)_+}\!\!
  \left(|f|_{W_p^k(\Omega)}\!+C_W\tau^{-1}\|f_X\|_2\right)\!+
  h^{-|\gamma|}\|r_L|X_L\|_\infty\right],
\end{equation*}
where $(t)_+=\max(t,0)$, the fill distance $h$ is assumed to be sufficiently small, the constant $C$ do not depend on $f,r_L$ or $h$, and the constant $C_W$ comes from Theorem \ref{SRTF:thm:WNbound}.
\end{thm}
\begin{proof}
According to the sampling inequality for functions from certain Sobolev spaces on a bounded domain (see Theorem $2.6$ in \cite{WendlandH2005_RBFSamplingInequalities}), we have
\begin{equation*}
  \|D^\gamma r_L\|_{L_q(\Omega_{L-1})}\!\leqslant\!
  C\left(h^{k-|\gamma|-\left(\frac{d}{p}-\frac{d}{q}\right)_+}
  |r_L|_{W_p^k(\Omega_{L-1})}\!+h^{-|\gamma|}\|r_L|X_L\|_\infty\right),
\end{equation*}
and further,
\begin{equation*}
  |r_L|_{W_p^k(\Omega_{L-1})}=|f-s_{\mathrm{SRT}}|_{W_p^k(\Omega_{L-1})}
  \leqslant|f|_{W_p^k(\Omega)}+|s_{\mathrm{SRT}}|_{W_p^k(\Omega_{L-1})}.
\end{equation*}
Applying the first inequality of Theorem \ref{SRTF:thm:WNbound} finishes the proof.
\end{proof}

This result also explains how the matrix $\Phi_i$ at each level affects the convergence. It is worth noting that this proof does not depend on the radial basis functions, so the next observation is an immediate consequence.
\begin{cor}
The result of Theorem \ref{SRTF:thm:EEW} holds for arbitrary basis functions based SRTs provided those basis functions belongs to $W_p^k(\Omega)$.
\end{cor}
It shows that a SRT, whose basis functions are differentiable and have bounded derivatives on $\Omega$ (regardless of polynomials, trigonometric polynomials, radial basis functions), leads to algebraic convergence orders for finitely smooth target functions. For infinitely smooth target functions, the following theorem shows that the Gaussian based SRT leads to exponential convergence orders.
\begin{thm}\label{SRTF:thm:EENG}
Under the supposition of Theorem \ref{SRTF:thm:Cbound}. If $f\in\mathcal{N}_{G_\delta}(\Omega)$ and $r_L$ is the residual $f-s_{\mathrm{SRT}}$ on an arbitrary leaf node $\Omega_{L-1}$, then for any $1\!\leqslant\! q\!\leqslant\!\infty$, $\gamma\in\mathbb{N}_0^d$, and $\delta>\delta_{L-1}$, there are constants $C$ and $h_0$ such that for all $h\leqslant h_0$, it holds that
\begin{equation*}
  \|D^\gamma r_L\|_{L_q(\Omega_{L-1})}\!\leqslant\!
  e^{C\log(h)/\sqrt{h}}\left(\|f\|_{\mathcal{N}_{G_\delta}(\Omega)}+
  C_{\mathcal{N}}\tau^{-1}\|f_X\|_2\right)
  +C'h^{-|\gamma|}\|r_L|X_L\|_\infty,
\end{equation*}
where the constant $C$ depends only on the geometry of $\Omega_{L-1}$, $h_0$ may depend on $d,p,q,\gamma$ and the geometry of $\Omega_{L-1}$ but not on $h$ or $f$, $C'$ do not depend on $h$ or $r_L$, and the constant $C_{\mathcal{N}}$ comes from Theorem \ref{SRTF:thm:WNbound}.
\end{thm}
\begin{proof}
According to the sampling inequality for functions from certain native spaces of Gaussians on a bounded domain (see Theorems $3.5$ and $7.5$ in \cite{RiegerC2010_RBFSamplingInequalities}), we have
\begin{equation*}
  \|D^\gamma r_L\|_{L_q(\Omega_{L-1})}\!\leqslant\!
  e^{C\log(h)/\sqrt{h}}\|r_L\|_{\mathcal{N}_{G_\delta}(\Omega_{L-1})}
  +C'h^{-|\gamma|}\|r_L|X_L\|_\infty,
\end{equation*}
and further,
\begin{equation*}
  \|r_L\|_{\mathcal{N}_{G_\delta}(\Omega_{L-1})}
  =\|f-s_{\mathrm{SRT}}\|_{\mathcal{N}_{G_\delta}(\Omega_{L-1})}
  \leqslant\|f|\Omega_{L-1}\|_{\mathcal{N}_{G_\delta}(\Omega_{L-1})}
  +\|s_{\mathrm{SRT}}\|_{\mathcal{N}_{G_\delta}(\Omega_{L-1})},
\end{equation*}
where $f|\Omega_{L-1}$ is the restriction of $f$ to $\Omega_{L-1}$ with $\|f|\Omega_{L-1}\|_{\mathcal{N}_{G_\delta}(\Omega_{L-1})}\leqslant
\|f\|_{\mathcal{N}_{G_\delta}(\Omega)}$ (see Theorem $10.47$ in \cite{WendlandH2005B_ScatteredData}); and applying the second inequality of Theorem \ref{SRTF:thm:WNbound} finishes the proof.
\end{proof}

Similarly, according to Remark \ref{SRTF:rem:M} and the sampling inequality for functions from certain native spaces of Gaussians on a bounded domain (see Theorems $3.5$ and $7.6$ in \cite{RiegerC2010_RBFSamplingInequalities}), we can also prove the convergence for the inverse multiquadric based SRTs.
\begin{thm}\label{SRTF:thm:EENM}
Under the supposition of Theorem \ref{SRTF:thm:Cbound}. If $f\in\mathcal{N}_{M_\delta}(\Omega)$, $s_{\mathrm{SRT}}$ is based on inverse multiquadrics, and $r_L$ is the residual $f-s_{\mathrm{SRT}}$ on an arbitrary leaf node $\Omega_{L-1}$, then for any $1\!\leqslant\! q\!\leqslant\!\infty$, $\gamma\in\mathbb{N}_0^d$, and $\delta>\delta_{L-1}$, there are constants $C$ and $h_0$ such that for all $h\leqslant h_0$, it holds that
\begin{equation*}
  \|D^\gamma r_L\|_{L_q(\Omega_{L-1})}\!\leqslant\!
  e^{-\frac{C}{\sqrt{h}}}\left(\|f\|_{\mathcal{N}_{M_\delta}(\Omega)}+
  C_{\mathcal{N}}\tau^{-1}\|f_X\|_2\right)
  +C'h^{-|\gamma|}\|r_L|X_L\|_\infty,
\end{equation*}
where the constants $C$ and $h_0>0$ depends only on $d,p,q,\gamma$ and the geometry of $\Omega_{L-1}$, $C'$ do not depend on $h$ or $r_L$, and the constant $C_{\mathcal{N}}$ comes from Theorem \ref{SRTF:thm:WNbound}.
\end{thm}

\subsection{Error estimates for SRFs}
\label{SRTF:s5:3}

For any $x\in\Omega_{L-1}\subset\Omega$, each SRT prediction $s_{\mathrm{SRT}}^{(i)}(x)$ ($1\leqslant i\leqslant n_t$) in a SRF converges to the target function $f(x)$ and satisfies relevant error estimates, thus, together with the Strong Law of Large Numbers and the Lindeberg-Levy central limit theorem, it follows that:
\begin{thm}
For any $x\in\Omega_{L-1}$, there exists an expectation $m_{\mathrm{SRF}}(x)$ such that
\begin{equation*}
  \lim_{n_t\to\infty}s_{\mathrm{SRF}}(x;n_t)
  =\lim_{n_t\to\infty}\left(\frac{1}{n_t}
  \sum_{i=1}^{n_t}s_{\mathrm{SRT}}^{(i)}(x)\right)
\end{equation*}
converges almost surely to $m_{\mathrm{SRF}}(x)$. Further, for any $1\leqslant q\leqslant\infty$ and $\gamma\in\mathbb{N}_0^d$, if $\|D^\gamma(s_{\mathrm{SRT}}^{(i)}(x)-f(x))\|_{L_q(\Omega_{L-1})}\leqslant\epsilon$, then there exists $\sigma\leqslant2\epsilon$ such that the random variables $\|D^\gamma(s_{\mathrm{SRF}}(x;n_t)- m_{\mathrm{SRF}}(x))\|_{L_q(\Omega_{L-1})}$ converge in distribution to a normal $N(0,\sigma/\sqrt{n_t})$, i.e., for any $\lambda_a>0$, the inequality
\begin{equation*}
  \left\|D^\gamma\Big(s_{\mathrm{SRF}}(x;n_t)-m_{\mathrm{SRF}}(x)
  \Big)\right\|_{L_q(\Omega_{L-1})}\leqslant\frac{\lambda_a\sigma}{\sqrt{n_t}}
  \leqslant\frac{2\lambda_a\epsilon}{\sqrt{n_t}}
\end{equation*}
holds with probability $1-a$, where $a=\frac{1}{\sqrt{2\pi}}\int_{-\lambda_a}^{\lambda_a}e^{-\frac{t^2}{2}}\ud t$.
\end{thm}

Obviously, the above result also holds for the SRF prediction defined in \eqref{SRTF:eq:SRF} that is more stable and is specially designed for overcoming the boundary effect of the error, as shown in Fig. \ref{SRTF:fig:3}. Combining the results of the previous subsection, one can obtain the error estimates for SRF predictions in the corresponding spaces.

\subsection{Complexity analysis}

Since the maximum depth of a binary tree is $\log_2N$ and the full data is only used for updating the residual, it is easy to see that:
\begin{thm}
Algorithm in section \ref{SRTF:s3} needs $\mathcal{O}(N\log_2N)$ time and $\mathcal{O}(N\log_2N)$ space in the worst case to train a SRT for $N$ arbitrary distributed points; and needs $\mathcal{O}(\log_2N)$ time in the worst case to make a prediction for a new point $x$. And the costs of algorithm in section \ref{SRTF:s4} are $n_t$ times that of the SRT for a SRF with $n_t$ SRTs.
\end{thm}

This result shows that the SRT or SRF also yields the excellent performance in terms of efficiency in addition to accuracy and adaptability. It is worth pointing out that the algorithm in section \ref{SRTF:s3} is designed to achieve hierarchical parallel processing so that the training process can be accelerated using multi-core architectures.

\section{Numerical examples}
\label{SRTF:s6}

In this section we compare the performance of both SRT and SRF with the Gaussian process regression (GPR). For an approximation $s$ of the target function $f$ on a test dataset $Z=\{z_i\}_{i=1}^{N_t}$ of size $N_t$, we use the relative mean absolute error (RMAE) as a measure of accuracy, i.e.,
\begin{equation}\label{SRTF:eq:RMAE}
  \mathrm{RMAE}=\frac{\sum_{i=1}^{N_t}|s(z_i)-f(z_i)|}{\sum_{i=1}^{N_t}|f(z_i)|}.
\end{equation}
We use two test functions: one is Franke's function, which is defined as:
\begin{align}\label{SRTF:eq:Franke}
f(x)=&\frac{3}{4}\exp\left(-\frac{(9x_1\!-\!2)^2}{4}-\frac{(9x_2\!-\!2)^2}{4}\right) +\frac{3}{4}\exp\left(-\frac{(9x_1+1)^2}{49}-\frac{9x_2+1}{10}\right) \\
&+\frac{1}{2}\exp\left(-\frac{(9x_1\!-\!7)^2}{4}-\frac{(9x_2\!-\!3)^2}{4}\right) \!-\!\frac{1}{5}\exp\left(-(9x_1\!-\!4)^2\!-(9x_2\!-\!7)^2\right)\nonumber,
\end{align}
where $x\in[0,1]^d$ for $d\geqslant2$; and the other is local oscillating and defined as:
\begin{align}\label{SRTF:eq:LO}
g(x)=&-2x_1x_2+2x_2^2-330\exp\left(-\frac{\|x\|_2^2}{2}\right)\sin(2\|x\|_2^2),
\end{align}
where $x\in[-7,7]^d$ for $d\geqslant2$.

All our numerical tests are based on scattered data which are either randomly generated or the Halton sequence \cite{HaltonJH1960M_HaltonSequence}. In addition, the procedure for the above two methods at each sample size is repeated $5000$ times for investigating the stability of the results. We use {\ttfamily Matlab}'s function {\ttfamily fitrgp} to generate a GPR model trained using the same sample data of proposed methods. Fit the GPR model using the subset of regressors method for parameter estimation and fully independent conditional method for prediction. Standardize the predictors. Besides, since the computational complexity of GPR is $\mathcal{O}(N^3)$ for training work and $\mathcal{O}(N^2)$ for each prediction, where $N$ is the sample size, it is very difficult to use GPR for large data set, so the sample size is varied from $10^1$ to $10^4$ for all numerical tests by using GPR.

\subsection{Accuracy, sparsity, storage and computational time}

The size of Halton points are varied from $10^1$ to $10^4$ for Franke's function. The results are shown in Fig. \ref{SRTF:fig:7}. From the upper left of Fig. \ref{SRTF:fig:7}, as expected, the RMAE of both SRT and SRF are much lower than GPR as data point $N$ is large. Moreover, from the upper right of Fig. \ref{SRTF:fig:7}, we can find out that the average number of centers for SRT prediction at $1$ point is varying from $10$ to $182$. Besides, since the size of sample points is varied from $10^1$ to $10^4$, both the storage requirement and the computational time of the proposed methods are much lower than those of GPR.

\begin{figure}[!htb]
\centering
\begin{minipage}{0.8\textwidth}
  \centering
  \includegraphics[width=0.4\textwidth]{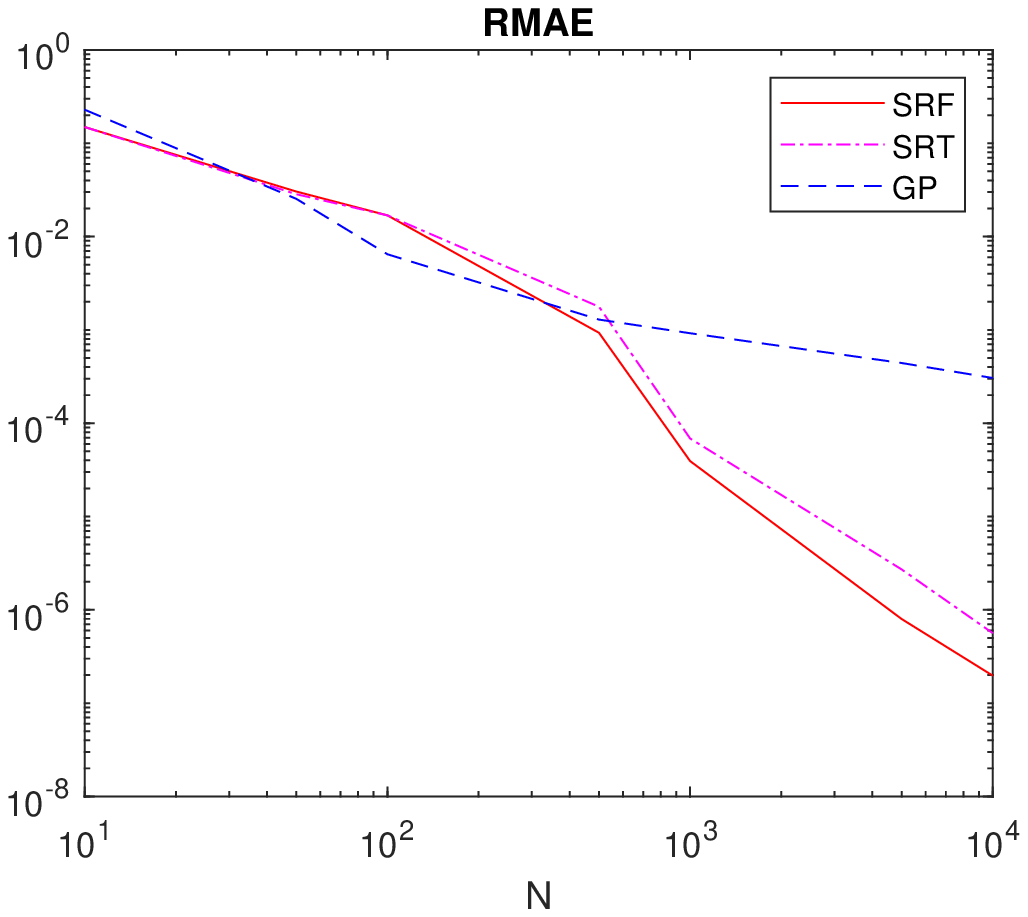}
  \includegraphics[width=0.4\textwidth]{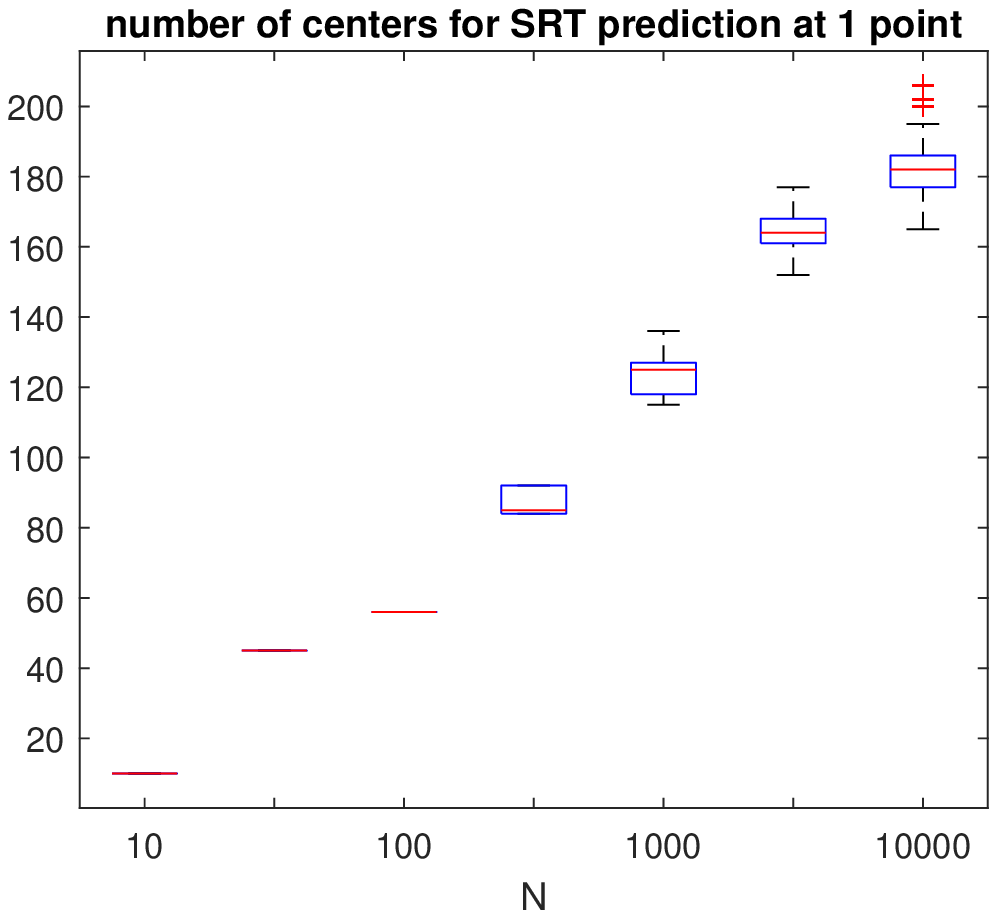}
\end{minipage}
\begin{minipage}{0.8\textwidth}
  \centering
  \includegraphics[width=0.4\textwidth]{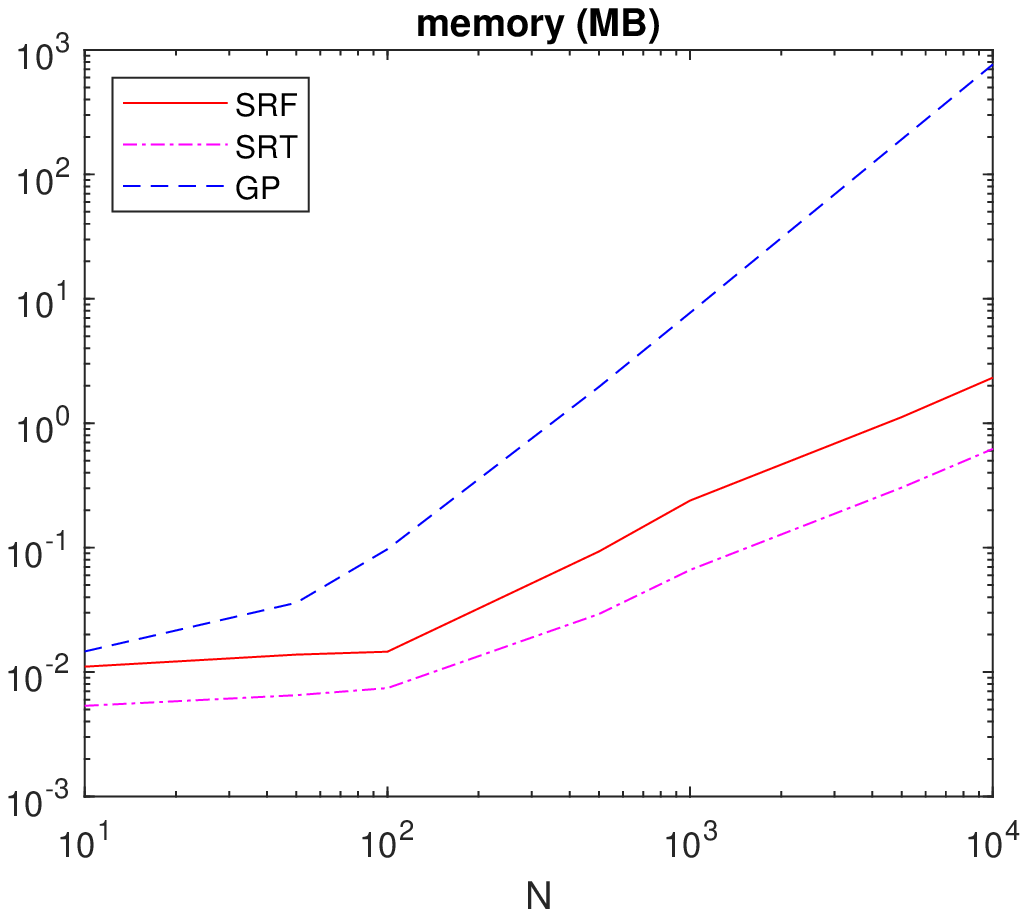}
  \includegraphics[width=0.4\textwidth]{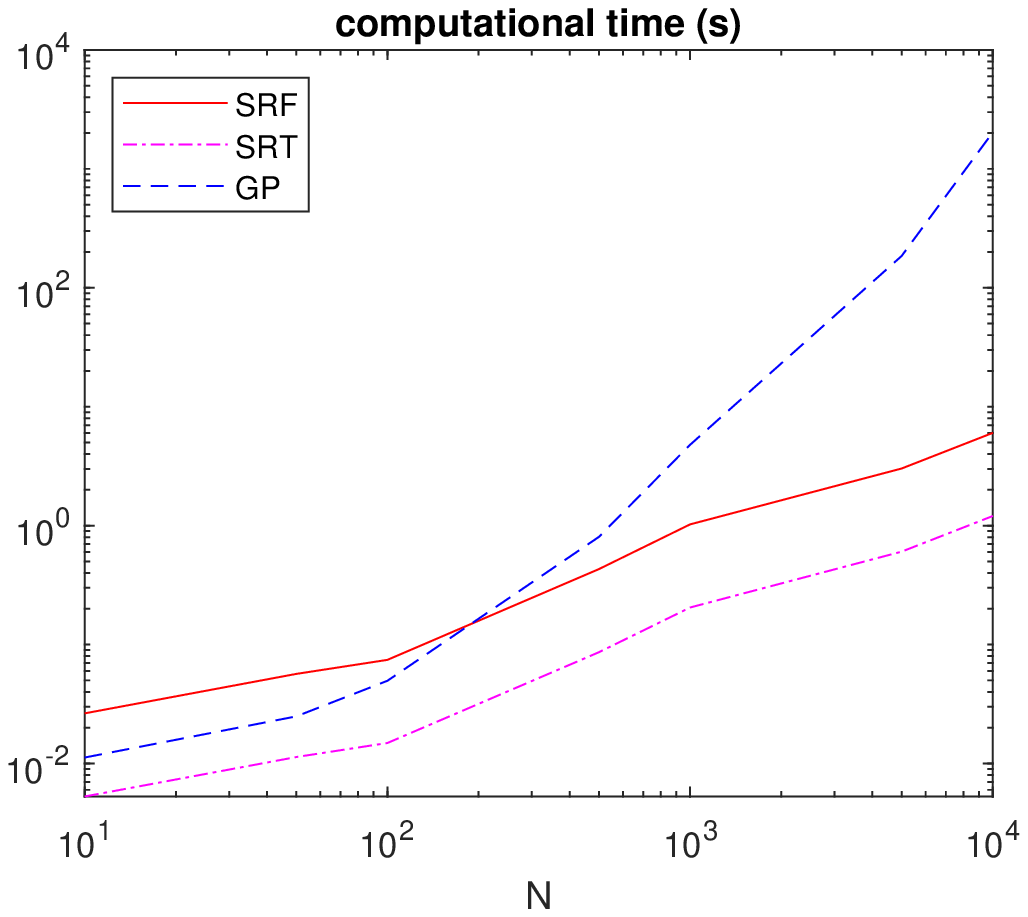}
\end{minipage}
\caption{\itshape\small{Results are shown for up to $N=10^4$ Halton points (with $N_t=5000$ test points different from the interpolation points) in $d=2$.}}
\label{SRTF:fig:7}
\end{figure}

\subsection{Insufficient data report}

We choose the second test function $g(x)$, $x\in[-7,7]^2$, to illustrate the insufficient data situation. From Fig. \ref{SRTF:fig:8} we can find out that since $g(x)$ is complicated near the central of domain and the RAE of residual still does not reach the expected error when the sample points $N=3000$; that is, the relevant node is lack of data at this area. Further, by adding the size of sample points $N$ to $6000$, as expected, from the lower left of Fig.\ref{SRTF:fig:8} we find out that the RAE clearly decreased (with the maximum RAE decreases from $0.1926$ to $0.0784$). Besides, from the lower right of Fig. \ref{SRTF:fig:8}, the median value of centers for SRT prediction at one point for both $N=3000$ and $N=6000$ are close to $110$.

\begin{figure}[!htb]
\centering
\begin{minipage}{0.8\textwidth}
  \centering
  \includegraphics[width=0.4\textwidth]{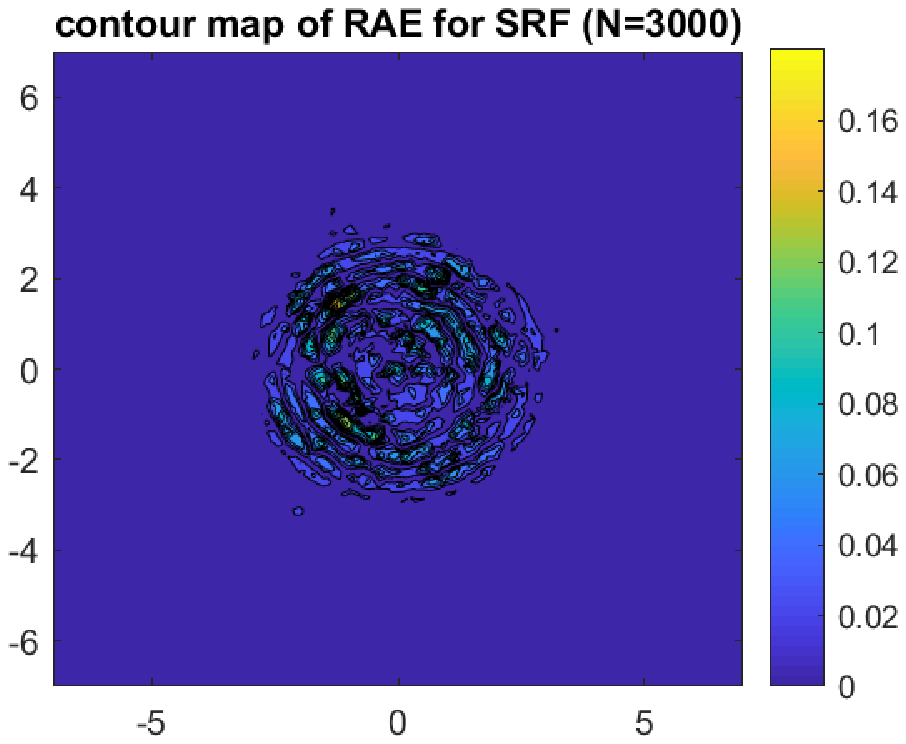}
  \includegraphics[width=0.4\textwidth]{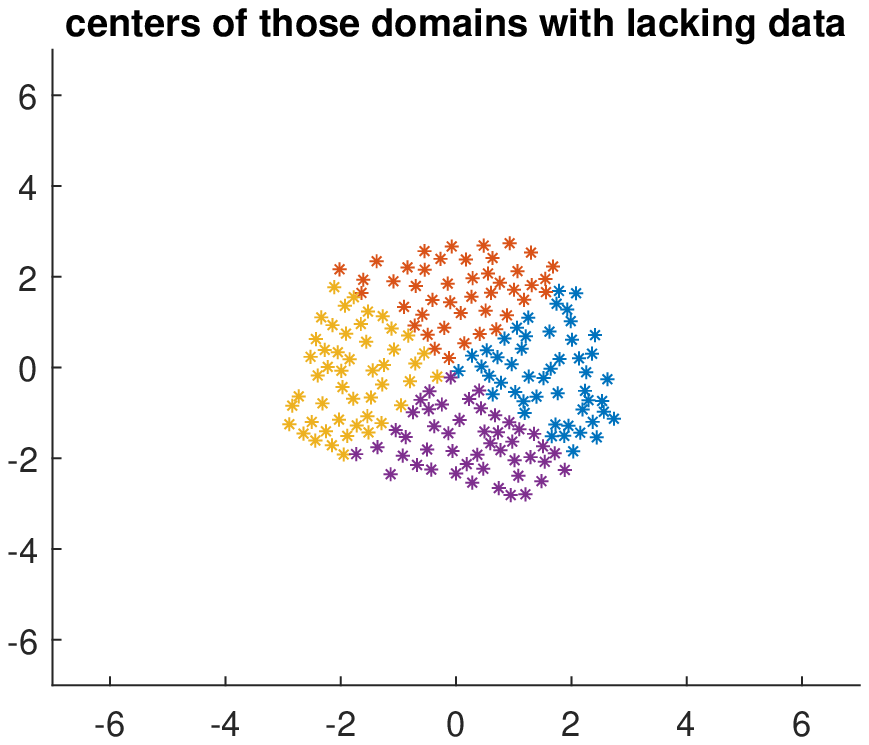}
\end{minipage}
\begin{minipage}{0.8\textwidth}
  \centering
  \includegraphics[width=0.4\textwidth]{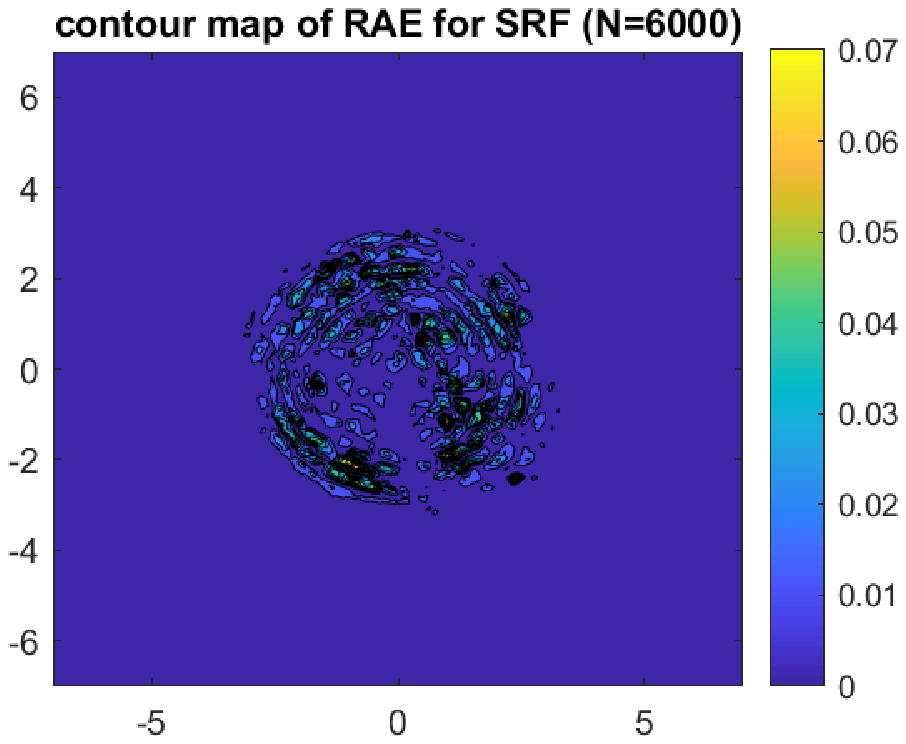}
  \includegraphics[width=0.4\textwidth]{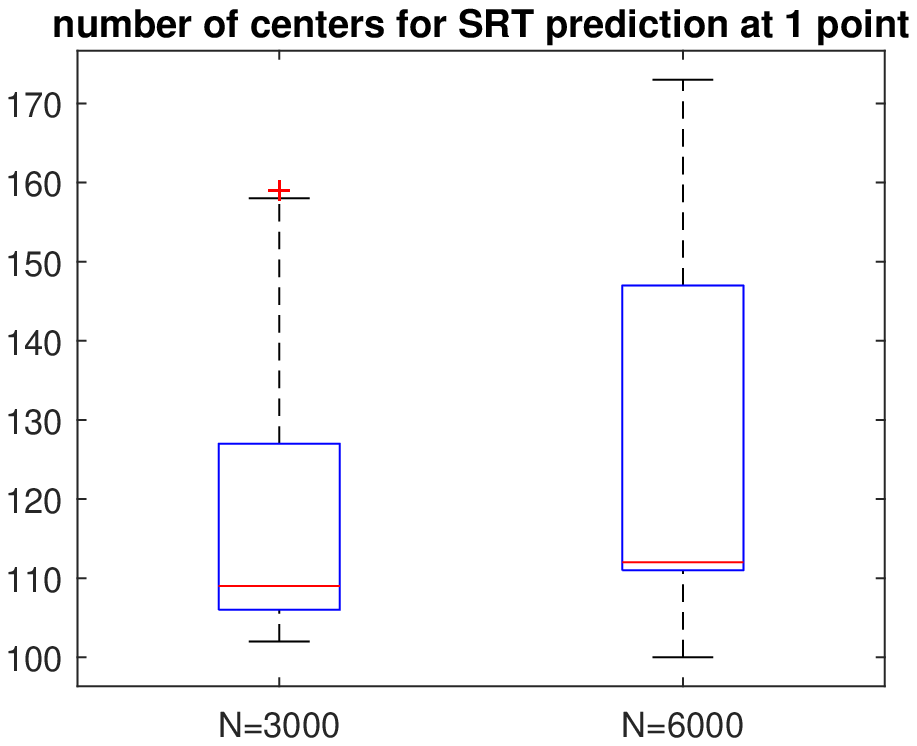}
\end{minipage}
\caption{\itshape\small{Results are shown for $N$ Halton points (with $N_t=10^4$ test points different from the interpolation points) in $d=2$.}}
\label{SRTF:fig:8}
\end{figure}

\subsection{$3$-dimensional problem}

The $3$-dimensional Franke's function can be shown in Fig. \ref{SRTF:fig:9}. The size of Halton points are varied from $10^1$ to $10^6$. We find out that the RMAEs of proposed methods are not as good as that of GPR when size of sample points is less than $10^4$. Further, since the size of sample points is varied from $10^4$ to $10^6$, leading to the value of error varying from $5.7224\times10^{-4}$ to $2.3126\times10^{-7}$ by using SRT, and from $1.3037\times10^{-4}$ to $4.7757\times10^{-8}$ by using SRF. Besides, from low figures of Fig. \ref{SRTF:fig:9}, it is noted that both the storage requirement and the computing time of proposed methods are less than those of GPR. The average number of centers for SRT prediction at one point is varying from $10$ to $853$.

\begin{figure}[!htb]
\centering
\begin{minipage}{0.8\textwidth}
  \centering
  \includegraphics[width=0.4\textwidth]{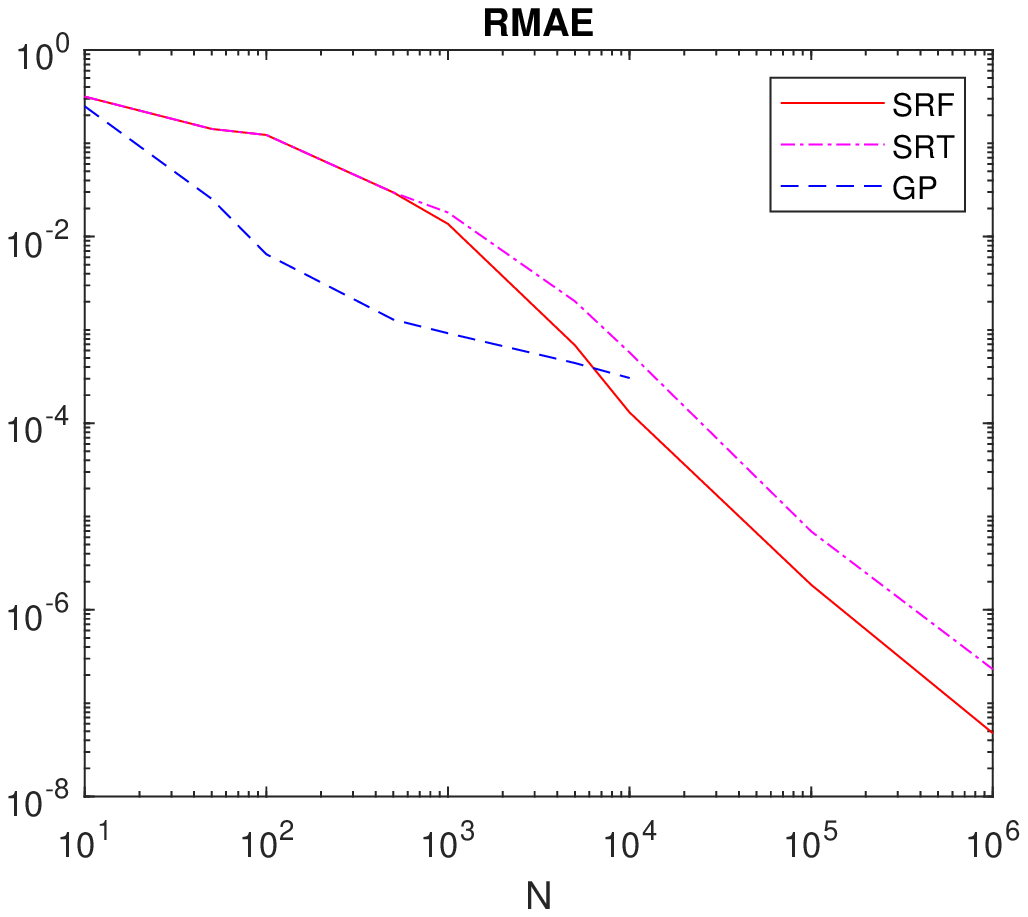}
  \includegraphics[width=0.4\textwidth]{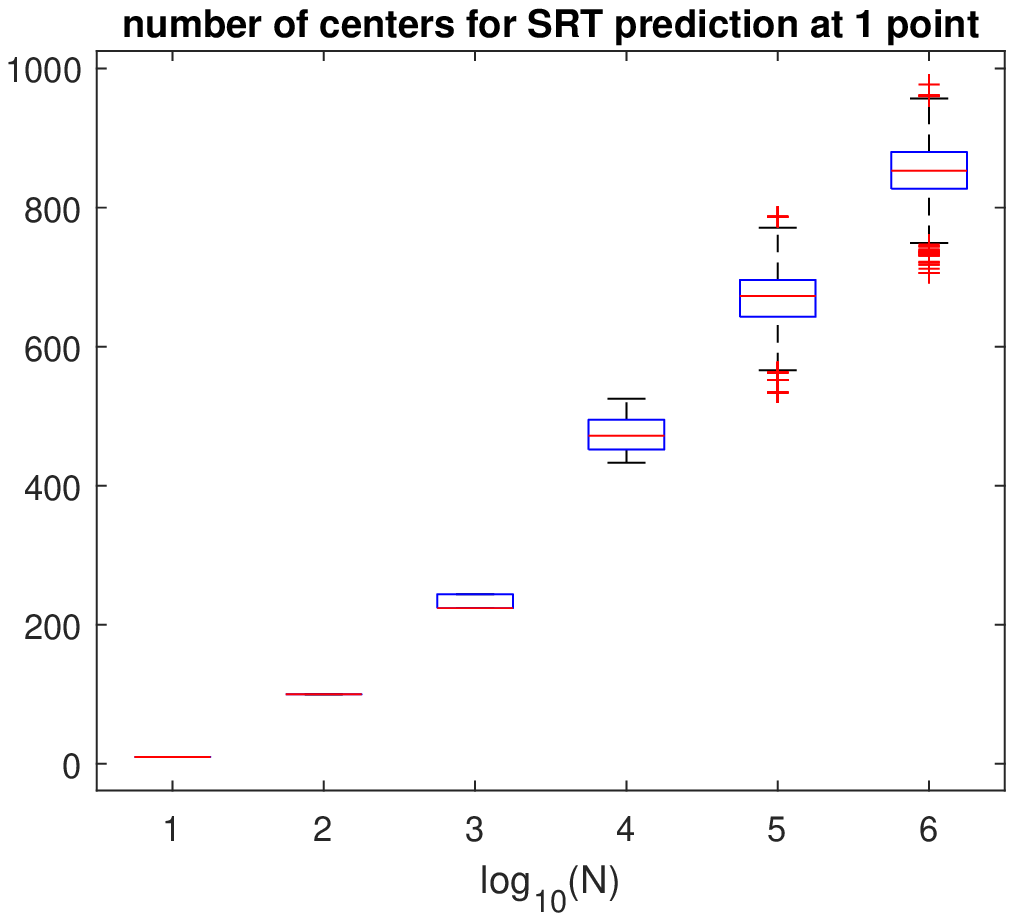}
\end{minipage}
\begin{minipage}{0.8\textwidth}
  \centering
  \includegraphics[width=0.4\textwidth]{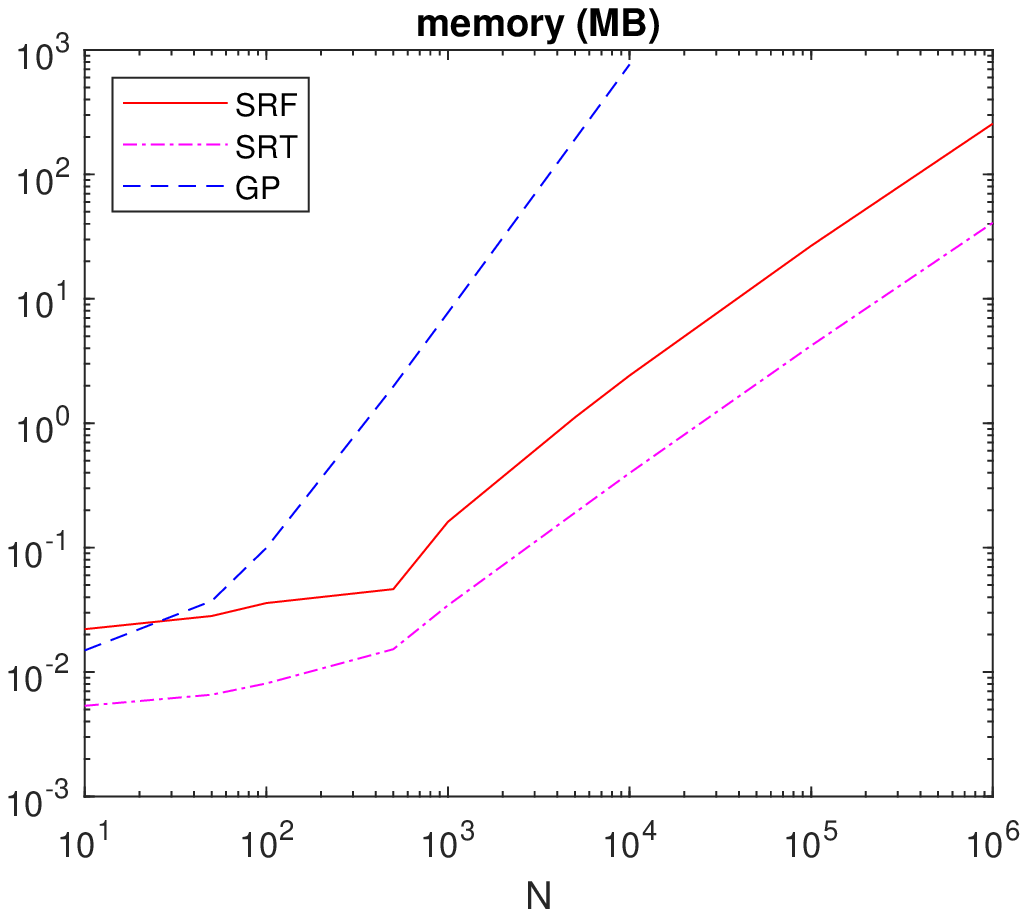}
  \includegraphics[width=0.4\textwidth]{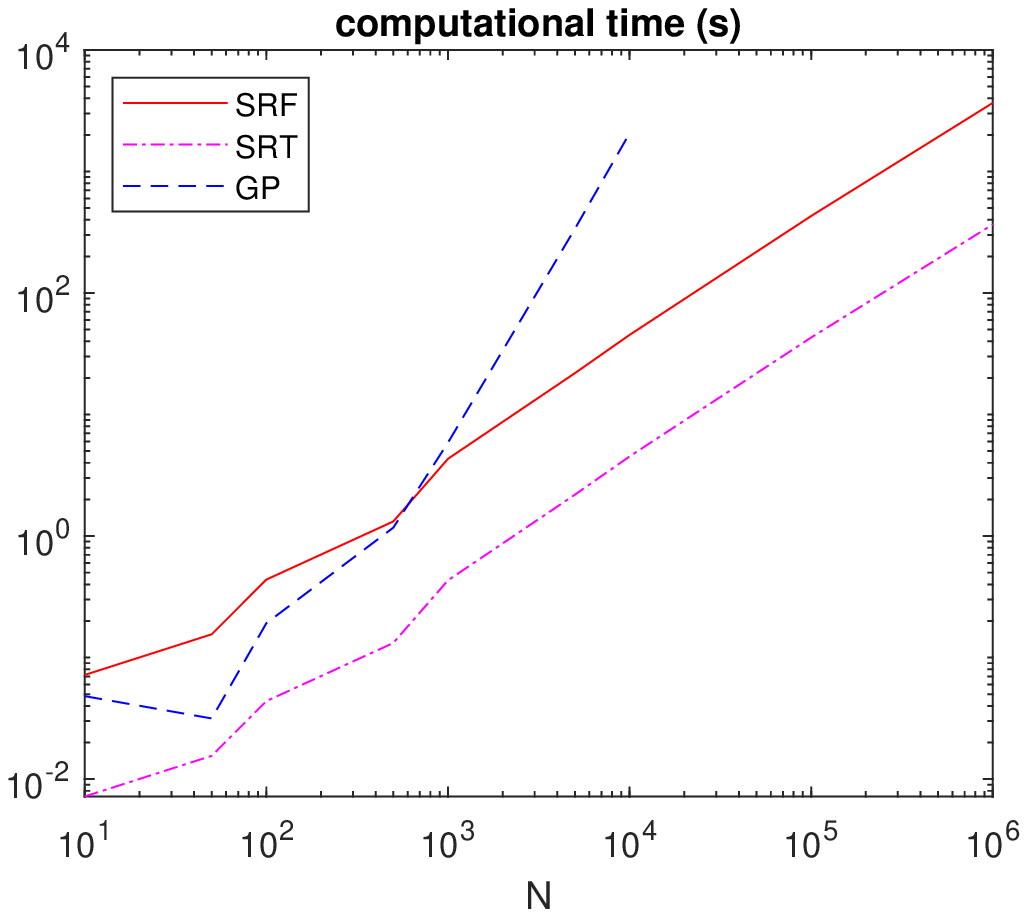}
\end{minipage}
\caption{\itshape\small{Results are shown for up to $N=10^6$ Halton points (with $N_t=5000$ test points different from the interpolation points) in $d=3$.}}
\label{SRTF:fig:9}
\end{figure}

\section{Conclusions}
\label{SRTF:s7}

In this work, we proposed two new methods for multivariate scattered data approximation, named Sparse residual tree (SRT) and Sparse residual tree (SRF), respectively. We proved that the time complexity of SRTs is less than $\mathcal{O}(N\log_2 N)$ for the initial work and $\mathcal{O}(\log_2N)$ for each prediction, and the storage requirement is less than $\mathcal{O}(N\log_2N)$, where $N$ is the data points. From the numerical experiments, we can find out that the proposed methods are good at dealing with cases where the data is sufficient or even redundant. For the higher dimensional problem, the proposed methods do not work as well as we expected. The possible reason is that the sample size is usually difficult to be sufficient or even redundant for higher dimensional problems, and the proposed methods tend to point out the possible local regions where data refinement is needed, rather than obtain approximations. It provides that the proposed methods can be used to solve the large data sets problems. In the following works, we will try to improve the proposed methods for solving higher dimensional problems.

\bibliographystyle{unsrtnat}
\bibliography{MReferences}
\pdfbookmark[0]{References}{section}

\end{document}